\documentclass[bibliography=totoc]{scrartcl}
\usepackage[T1]{fontenc}
\usepackage[utf8]{inputenc}
\usepackage[]{amsmath,amssymb, amsthm}
\usepackage{graphicx}
\usepackage{subfigure}
\usepackage{multirow}
\usepackage[round]{natbib}
\usepackage{authblk}
\usepackage{enumerate}
\usepackage{latexsym}
\begin{document}
\title{Generalized linear statistics for near epoch dependent processes with application to EGARCH-processes}
\author{Svenja Fischer 
 \thanks{Institute of Hydrology, Ruhr-Universit\"at Bochum,
 	D-44801 Bochum, Germany,
 	\texttt{svenja.fischer@rub.de}}}

\date{}
\maketitle
\bibliographystyle{plainnat}

\setlength{\parindent}{0pt}
\allowdisplaybreaks

\newtheorem{defin}{Definition}[section]
\newtheorem{theo}{Theorem}[section]
\newtheorem{coro}{Corollary}[section]
\newtheorem{lemm}{Lemma}[section]
\newtheorem{prop}{Proposition}[section]
\newtheorem{ass}{Assumption}[section]

\theoremstyle{definition}
\newtheorem{Exam}{Example}[section]
\newtheorem{Rem}{Remark}[section]

\newcommand{\var}{\operatorname{Var}}
\newcommand{\cov}{\operatorname{Cov}}
\newcommand{\med}{\operatorname{med}}

\begin{abstract}
\noindent
The class of Generalized L-statistics ($GL$-statistics) unifies a broad class of different estimators, for example  scale estimators based on multivariate kernels. $GL$-statistics are functionals of $U$-quantiles and therefore the dimension of the kernel of the $U$-quantiles determines the kernel dimension of the estimator. Up to now only few results for multivariate kernels are known. Additionally, most theory was established under independence or for short range dependent processes. In this paper we establish a central limit theorem for $GL$-statistics of functionals of short range dependent data, in perticular near epoch dependent sequences on absolutely regular processes, and arbitrary dimension of the underlying kernel. This limit theorem is based on the theory of $U$-statistics and $U$-processes, for which we show a central limit theorem as well as an invariance principle. The usage of near epoch dependent processes admits us to consider functionals of short range dependent processes and therefore models like the EGARCH-model.
We also develop a consistent estimator of the asymptotic variance of $GL$-statistics.

\end{abstract}

\section{Introduction and Examples}
The class of Generalized linear statistics ($GL$-statistics) is known to unify some of the most common classes of statistics, such as $U$-statistics and $L$-statistics. They are defined as functionals of $U$-quantiles and the theory of $U$-statistics and $U$-processes proves itself to be a key tool in handling the asymptotic of $GL$-statistics. \cite{Serf1984} already showed the asymptotic normality by an approximation via $U$-statistics. This result was developed under the assumption of independence of the underlying random variables. For short range dependence \cite{Fischer.2016} showed the validity of the asymptotic normality. They used the concept of strong mixing, the weakest form of mixing. Nevertheless, this form of short range dependence does not contain several common used models. For example, GARCH-models or dynamical systems are excluded. These are models which are functionals of short range dependent data. We therefore want to consider the concept of near epoch dependent data on mixing processes in this paper, which includes such models.

\vspace{0.3cm}
First we want to state some general assumptions.

\vspace{0.2cm}

Let $X_1,\ldots, X_n$ be a sequence of random variables with distribution function $F$. As mentioned above these random variables shall not be independent but functionals of mixing sequences. A detailed definition is given later on. Moreover, let $F_n$ be the empirical distribution function of $X_1,\ldots,X_n$ given by
\begin{align*}
	F_n(x)=\frac{1}{n}\sum_{i=1}^{n}{1_{\left[X_i\leq x\right]}}, ~-\infty<x<\infty,
\end{align*}
and let $h(x_1,\ldots, x_m)$ be a kernel (a measurable and symmetric function) with given dimension $m\geq 2$. The related empirical distribution function $H_n$ of $h\left(X_{i_1},\ldots, X_{i_m}\right)$ is given by
\begin{align*}
	H_n(x)=\frac{1}{\binom{m}{n}}\sum_{1\leq i_1<\ldots < i_m\leq n}{1_{\left[h\left(X_{i_1},\ldots , X_{i_m}\right)\leq x\right]}}, ~ -\infty<x<\infty.
\end{align*}
$H_F$ is defined as the distribution function of the kernel $h$ with 

$H_F(y)=\mathbb{P}_F(h(Y_1,\ldots,Y_m)\leq y)$ for independent copies $Y_1,\ldots,Y_m$ of $X_1$ and $0<h_F<\infty$ the related density.

We define $h_{F;X_{i_2},\ldots,X_{i_k}}$ as the density of $h(Y_{i_1},X_{i_2},\ldots,X_{i_k},Y_{i_{k+1}},\ldots,Y_{i_m})$ for $2\leq k \leq m$ and $i_1<i_2<\ldots<i_m$.
\vspace*{0.3cm}

The object of interest in our work are $GL$-statistics.
$\newline$
A Generalized $L$-statistic with kernel $h$ is defined as

\begin{align*}
T(H_n)&=\int_0^1{H_n^{-1}(t)J(t)dt}+\sum_{i=1}^d{a_iH_n^{-1}(p_i)}\\
&=\sum_{i=1}^{n_{(m)}}{\left[\int_{\frac{(i-1)}{n_{(m)}}}^{\frac{i}{n_{(m)}}}{J(t)dt}\right]H_n^{-1}\left(\frac{i}{n_{(m)}}\right)}+\sum_{i=1}^d{a_iH_n^{-1}(p_i)}.
\end{align*}

The $GL$-statistic $T(H_n)$ estimates $T(H_F)$. For further information on $GL$-statistics we refer to  \citet{Serf1984} or \cite{Fischer.2016}.

In the following we present some common estimators that can be written as $GL$-statistic.
Besides the examples given in \cite{Fischer.2016} (generalized Hodges-Lehmann estimator, $\alpha$-trimmed mean and Generalized Median estimator) or \cite{Serf1984} there are some other estimators used for the estimation of the variance that can be expressed by a $GL$-statistic.

\vspace*{0.5cm}
\begin{Exam}
A robust estimator for the variance is Gini's Mean Difference. In contrast to the mean deviation this estimator is not only robust but also almost as efficient as the classical estimator for variance, the standard deviation (\cite{Gerst.2015}). It is given by
\begin{align*}
G_n=\frac{1}{n(n-1)}\sum_{i,j=1}^n \vert X_i-X_j \vert
\end{align*}
and can be written as a $GL$-statistic by choosing the kernel $h(x_i,x_j)=\vert x_i-x_j \vert$ (the kernel dimension is then of course $m=2$) and a constant continuous function $J(t)=1$. The discrete part of the $GL$-statistic vanishes by choosing $d=0$.
When considering the version of Gini's Mean Difference using order statistics, that is
\begin{align*}
G_n=\frac{2}{n(n-1)}\sum_{i=1}^n (2i-n-1)X_{(i:n)}
\end{align*}
with $X_{(i:n)}$ being the \textbf{i}th order statistic of the sample $X_1,\ldots,X_n$, the kernel is chosen as identity and $J(t)=\frac{4n}{n-1}t-\frac{2n}{n-1}$. 
\end{Exam}

The following two examples of scale estimators can be found in \cite{Croux.1992} and can all be expressed as a $GL$-statistic.

\begin{Exam}
An estimator for the scale which is robust with a breakdown point of $50\%$ is given by
\begin{align*}
Q=\med\limits_{i<j<k}\lbrace\min(\vert X_i-X_j\vert,\vert X_i-X_k\vert, \vert X_j-X_k\vert)\rbrace .
\end{align*}
It can be expressed as a $GL$-statistic by choosing the three-dimensional kernel $h(x_1,x_2,x_3)=\min(\vert x_1-x_2\vert, \vert x_1-x_3\vert, \vert x_2-x_3\vert )$ and the parameters $J=0, d=1, a_1=1$ and $p_1=1/2$. To generalise this estimator to subsamples of order greater than three and other quantiles the formula
\begin{align*}
Q_n^\alpha=\lbrace\min(\vert X_{i_l}-X_{i_k}\vert,1\leq l< k\leq m),1\leq i_1<\ldots<i_m\leq n\rbrace_{([\alpha\binom{n}{m}])},
\end{align*}
can be used, where $([\alpha \binom{n}{m}])$ denotes the empirical $\alpha$-quantile $(\alpha\in (0,1))$. Then 
 we can choose for arbitrary size $m$ of the subsample 
$h(x_1,\ldots,x_m)=\min(\vert x_j-x_i\vert, 1\leq i<j\leq m)$ with $J, d$ and $a_1$ as before but $p_1=\frac{[\alpha\binom{n}{m}]}{\binom{n}{m}}$.
\end{Exam}

\begin{Exam}
	Another location-free scale estimator is given by
	\begin{align*}
	C_n^\alpha=c_\alpha\vert X_{(i+[\alpha n]+1)}-X_{(i)}\vert_{([n/2]-[\alpha n])},
		\end{align*}
		$\alpha\in (0,0.5)$, which takes the $[n/2]-[\alpha n]$ order statistic of the difference of the first and last order statistic of all (sorted) subsamples of length $[\alpha n]+2$. The constant $c_\alpha$ makes the estimator Fisher-consistent under normality.
		 By the choice of a kernel $h$ of dimension $m=[\alpha n]+2$ with $h(x_1,\ldots,x_m)=\max(x_1,\ldots,x_m)-\min(x_1,\ldots,x_m)$ and $J=0, d=1, a_1=c_\alpha$ and $p_1=\frac{1}{\binom{n}{m}}$ the representation by a $GL$-statistic can be obtained.
		 A well-known special case of this estimator is the Least Median of Squares
		 \begin{align*}
		 LMS_n=0.7413\min\limits_i\vert X_{(i+[n/2])}-X_{(i)}\vert ,
		 \end{align*}
		 with $[\alpha n]=[n/2]-1$ and $c_\alpha=\frac{1}{2\Phi^{-1}(0.75)}=0.7413$.
\end{Exam}

 We want to emphasize that both estimators of \cite{Croux.1992} ($Q_n^\alpha$ and $	C_n^\alpha$) use multivariate kernels, that are kernels with dimension greater than 2. In these cases results for bivariate $U$-statistics would not be sufficient.

\vspace{0.3cm}

Let us now state the assumptions to consider functionals of short range dependent data. We want to consider absolutely regular random variables as underlying process. Absolutely regular (or $\beta$-mixing) is a stronger assumption than strong mixing since for the strong mixing coefficients $\alpha$ it is $\alpha(l)\leq \frac{1}{2}\beta(l)$ so that every absolutely regular process is likewise strong mixing. For more details see \cite{Brad2007}.

\begin{defin}
Let $\mathcal{A}, \mathcal{B} \subset \mathcal{F}$ be two $\sigma$-fields on the probability space $(\Omega, \mathcal{F},\mathbb{P})$. The absolute regularity coefficient of $\mathcal{A}$ and $\mathcal{B}$ is given by
\begin{align*}
\beta(\mathcal{A},\mathcal{B})=\mathbb{E}\sup\limits_{A\in \mathcal{A}}\left\vert \mathbb{P}(A\vert\mathcal{B})-\mathbb{P}(A)\right\vert.
\end{align*}
If $(X_n)_{n\in \mathbb{N}}$ is a stationary process, then the absolute regularity coefficients of $(X_n)_{n\in\mathbb{N}}$ are given by
\begin{align*}
\beta(l)=\sup\limits_{n\in\mathbb{N}}\mathbb{E}\sup\limits_{A\in \mathcal{F}_1^n}\left\vert \mathbb{P}(A\vert\mathcal{F}_{n+l}^\infty)-\mathbb{P}(A)\right\vert.
\end{align*}
$(X_n)_{n\in\mathbb{N}}$ is called absolutely regular, if $\beta(l)\rightarrow 0$ as $l\rightarrow \infty$.
\end{defin}

To consider not only short range dependent data but also functionals of these we use the concept of near epoch dependence. There exist different versions of near epoch dependence, for example  $L_p$- and P-near epoch dependence. Both concepts are applied widely when considering functionals of data, see for example \cite{Boro}, \cite{Dehling.2015} or \cite{Vogel.2015}. Often also the analogous definition of 1-approximating functionals is used. Here, we want to use the concept of $L_1$-near epoch dependence.

\begin{defin}
Let $((X_n,Z_n))_{n\in\mathbb{Z}}$ be a stationary process. $(X_n)_{n\in \mathbb{N}}$ is called $L_1$ near epoch dependent (NED) on the process $(Z_n)_{n\in \mathbb{Z}}$ with approximation constants $(a_l)_{l\in\mathbb{N}}$, if
\begin{align*}
\mathbb{E}\left\vert X_1-\mathbb{E}\left(X_1\vert\mathcal{G}_{-l}^l\right)\right\vert\leq a_l, \qquad l=0,1,2,\ldots,
\end{align*}
where $\lim\limits_{l\rightarrow \infty}a_l=0$ and $\mathcal{G}_{-l}^l$ is the $\sigma$-field generated by  $Z_{-l},\ldots,Z_l$.
\end{defin}

In the proofs we will see later on that this definition of near epoch dependence also implies P-near epoch dependence under certain conditions.

\vspace*{0.2cm}
With these assumptions we are now able to state a Central Limit Theorem for multivariate $U$-statistics as well as an Invariance Principle for $U$-processes in Section 2. These are the key tools for proving the asymptotic normality and the consistency of the long-run variance estimator for $GL$-statistics (Section 3), where the proofs are given in Section 4. Finally, we show that EGARCH-processes are near epoch dependent, which is used in a short simulation study on asymptotic normality.
\vspace*{0.2cm}

\section{$U$-statistics and $U$-processes}
The theory of $U$-statistics plays the most important role in proving our main theorem, the asymptotic normality of $GL$-statistics. $U$-statistics are used as an approximation of the error term $T(H_n)-T(H_F)$ and therefore the results of this section are needed for the main proof later on.

A $U$-statistic with kernel $h(x_1,\ldots,x_m)$ is given by
\begin{align*}
U_n=\frac{1}{\binom{n}{m}}\sum_{1\leq i_1<\ldots<i_m\leq n}h(X_{i_1},\ldots,X_{i_m}).
\end{align*}

For most of the results in this section we need a regularity condition for the kernel $h$. It is very similar to the Lipschitz-continuity and was developed by \cite{Denk1986}. The same variation condition is also used in \cite{Fischer.2016}.

\begin{defin}
	A kernel $h$ satisfies the variation condition, if there exists a constant $L$ and an $\epsilon_0>0$, such that for all $\epsilon\in (0,\epsilon_0)$
	\begin{align*}
	\mathbb{E}\left(\sup\limits_{\lVert(x_1,\ldots,x_m)-(X_1',\ldots,X_m')\rVert\leq \epsilon}\left| h(x_1,\ldots,x_m)-h(X_1',\ldots,X_m')\right|\right)\leq L\epsilon,
	\end{align*}
	where $X_i'$ are independent with the same distribution as $X_i$ and $\lVert \cdot \rVert$ is the Euklidean norm. 
	A kernel $h$ satisfies the extended variation condition, if there additionally exists a constant $L'$ and a $\delta_0>0$, such that for all $\delta\in(0,\delta_0)$ and all $2\leq k \leq m$
	\begin{align*}
	\mathbb{E}\left(\sup\limits_{\lvert x_{i_1}-Y_{i_1}\rvert\leq \delta}\left| h(x_{i_1},X_{i_2},\ldots,X_{i_k},Y_{i_{k+1}},\ldots,Y_{i_m})-h(Y_{i_1},X_{i_2},\ldots,X_{i_k},Y_{i_{k+1}},\ldots,Y_{i_m})\right|\right)\\
	\leq L'\delta
	\end{align*} 
	for independent copies $(Y_n)_{n\in\mathbb{N}}$ of $(X_n)_{n\in\mathbb{N}}$ and all $i_1<i_2<\ldots <i_m$. 
	If the kernel has dimension one, we note that it satisfies the extended variation condition, if it satisfies the variation condition.
\end{defin}

\begin{Rem}
	A Lipschitz-continuous kernel satisfies the variation condition.
\end{Rem}

We also need a second variation condition which demands regularity in the $L_2$-space.
\begin{defin}
	A kernel $h$ satisfies the $L_2$-variation condition, if there exists a constant $L$ and an $\epsilon_0>0$, such that for all $\epsilon\in (0,\epsilon_0)$
	\begin{align*}
	\mathbb{E}\left(\sup\limits_{\lVert(x_1,\ldots,x_m)-(X_1',\ldots,X_m')\rVert\leq \epsilon}\left| h(x_1,\ldots,x_m)-h(X_1',\ldots,X_m')\right|\right)^2\leq L\epsilon,
	\end{align*}
	where $X_i'$ are independent with the same distribution as $X_i$ and $\lVert \cdot \rVert$ is the Euklidean norm. 
\end{defin}

\begin{Rem}
	A bounded kernel which fulfils the variation condition also satisfies the $L_2$-variation condition due to the inequality $(a-b)^2\leq \vert a-b\vert\cdot (\vert a\vert +\vert b \vert)$.
\end{Rem}

A common technique when developing asymptotic results for $U$-statistics has benn developed by \cite{Hoeff1948} and makes a separate consideration of the single terms possible.

\begin{defin}\label{hoeff}(Hoeffding-Decomposition)
	$\newline$
	Let $U_n$ be a $U$-statistic with kernel $h=h(x_1,\ldots,x_m)$. Then we can write $U_n$ as
	\begin{align*}
	U_n=\theta+\sum_{j=1}^{m}{\binom{m}{j}\frac{1}{\binom{n}{j}}S_{jn}},
	\end{align*}
	where 
	\begin{align*}
	\theta&=\mathbb{E}(h(Y_1,\ldots,Y_m))\\
	\tilde{h}_j(x_1,\ldots,x_j)&=\mathbb{E}(h(x_1,\ldots,x_j,Y_{j+1},\ldots,Y_m))-\theta\\
	S_{jn}&=\sum_{1\leq i_1<\ldots<i_j\leq n}g_j(X_{i_1},\ldots,X_{i_j})\\
	g_1(x_1)&=\tilde{h}_1(x_1)\\
	g_2(x_1,x_2)&=\tilde{h}_2(x_1,x_2)-g_1(x_1)-g_1(x_2)\\
	g_3(x_1,x_2,x_3)&=\tilde{h}_3(x_1,x_2,x_3)-\sum_{i=1}^{3}g_1(x_i)-\sum_{1\leq i<j\leq 3}g_2(x_i,x_j)\\
	&\ldots\\
	g_m(x_1,\ldots,x_m)&=\tilde{h}_m(x_1,\ldots,x_m)-\sum_{i=1}^{m}g_1(x_i)-\sum_{1\leq i_1<i_2\leq m}g_2(x_{i_1},x_{i_2})\\
	&-\cdots - \sum_{1\leq i_1<\ldots<i_{m-1}\leq m}{g_{m-1}(x_{i_1},\ldots,x_{i_{m-1}})}.
	\end{align*}
	for independent copies $Y_1,\ldots,Y_m$ of $X_1$.
\end{defin}

We call $\frac{m}{n}\sum_{i=1}^ng_1(X_i)$ the linear part, the remaining parts are called degenerated.

\cite{Fischer.2016} already have shown that if the kernel $h$ satisfies the (extended) variation condition then also the Hoeffding kernels $g_1,\ldots,g_m$ do.

\vspace*{0.5cm}

The following theorem on the asymptotic normality of $U$-statistics for NED random variables is well-known for bivariate kernels (\cite{Wen2011}), but our theorem admits arbitrary dimension $m$. Under independence one can find a central limit theorem for $U$-statistics for example in \cite{serf1980}, whereas \cite{Wen2011} and \cite{Fischer.2016} show the same result for strong mixing random variables.

\begin{theo}\label{asynom}
	$\newline$
	Let $h:\mathbb{R}^m\rightarrow \mathbb{R}$ be a bounded kernel satisfying the extended variation condition.
	 Moreover, let $(X_n)_{n\in\mathbb{N}}$ be $L_1$ NED with approximation constants $(a_l)_{l\in\mathbb{N}}$ on an absolutely regular process $(Z_n)_{n\in\mathbb{Z}}$ with mixing coefficients $(\beta(l))_{l\in\mathbb{N}}$. Assume that there exists a $\delta>1$, such that $\beta(l)=O\left(l^{-\delta}\right)$ and $a_l=O\left(l^{-\delta-2}\right)$.
	Then we have
	\begin{align*}
	\sqrt{n}(U_n-\theta)\stackrel{D}{\longrightarrow}N(0,m^2\sigma^2)
	\end{align*}
	with $\sigma^2=\var(g_1(X_1))+2\sum_{j=1}^{\infty}{\cov(g_1(X_1),g_1(X_{1+j}))}$.
	
	If $\sigma=0$, then the statement is meant as convergence to $0$.
\end{theo}

In general, if the distribution of the $(X_n)_{n\in \mathbb{N}}$ is not specified, the long run variance $\sigma^2$ in Theorem \ref{asynom} is unknown. Therefore, for applications an estimator of $\sigma^2$ is needed. For bivariate $U$-statistics or $U$-processes \cite{Dehling.2015} and \cite{Vogel.2015} give consistent estimators by using an empirical version of the first Hoeffding kernel and a weight function. The multivariate extension to this estimator is
\begin{align*}
\hat{\sigma}^2=\sum_{r=-(n-1)}^{n-1} \kappa \left(\frac{\vert r \vert}{b_n}\right)\hat{\rho}(r),
\end{align*}
where $\kappa$ is the weight function and $b_n$ the bandwidth. 
\begin{align*}
\hat{\rho}(r)=\frac{1}{n}\sum_{i=1}^{n-r}\hat{g}_1(X_i)\hat{g}_1(X_{i+r})
\end{align*}
is the empirical covariance for lag $r$, using the empirical version of the first Hoeffding kernel
\begin{align*}
\hat{g}_1(x)=\frac{1}{n^{m-1}}\sum_{1\leq i_1<\ldots<i_{m-1}\leq n}h(x,X_{i_1},\ldots,X_{i_{m-1}})-\frac{1}{n^m}\sum_{1\leq i_1<\ldots <i_m<n}h(X_{i_1},\ldots,X_{i_m}).
\end{align*}

As \cite{Dehling.2015} have already shown we need some regularity conditions for $\kappa$ and $b_n$ to achieve consistency of the estimator. These are similar to the assumption made in \cite{deJong.2000} and are fulfilled, for example, by the Bartlett kernel $\kappa(t)=(1-\vert t \vert )1_{[\vert t \vert \leq 1]}$.

\begin{ass}\label{asskern}
	The function $\kappa:[0,\infty)\rightarrow [0,1)$ is continuous at 0 and all but a finite number of points. Moreover, $\vert \kappa \vert$ is dominated by a non-increasing, integrable function and \begin{align*}
	\int_0^\infty\Big\vert \int_0^\infty \kappa(t) \cos(xt)dt \Big\vert dx < \infty.
	\end{align*}
	The bandwith $b_n$ satisfies $b_n\rightarrow \infty$ as $n\rightarrow \infty$ and $b_n/\sqrt{n} \rightarrow 0$.
\end{ass}
With this considerations we are able to show that $\hat{\sigma}^2$ is a consistent estimator for the long-run variance.
\begin{theo}\label{kernest}
	Let $h:\mathbb{R}^m\rightarrow \mathbb{R}$ be a bounded kernel satisfying the extended variation condition and the $L_2$-variation condition.
	
	Moreover, let $(X_n)_{n\in\mathbb{N}}$ be NED with approximation constants $(a_l)_{l\in\mathbb{N}}$ on an absolutely regular process $(Z_n)_{n\in\mathbb{Z}}$ with mixing coefficients $(\beta(l))_{l\in\mathbb{N}}$ and let a $\delta>11$  exist, such that $\sum_{l=1}^\infty l\beta^{2/(2+\delta)(l)}<\infty$ and $a_l=O\left(l^{-\delta-3}\right)$. The weight function $\kappa$ and the bandwidth $b_n$ should fulfil Assumption \ref{asskern}.
	Then
	\begin{align*}
	\hat{\sigma}^2\rightarrow \sigma^2 \text{ for } n\rightarrow \infty,
	\end{align*}
	where $\sigma^2=\var(g_1(X_1))+2\sum_{j=1}^{\infty}{\cov(g_1(X_1),g_1(X_{1+j}))}$.
\end{theo}

\vspace*{0.5cm}

In the following we no longer want to consider $U$-statistics but $U$-processes. That is we have a process $(U_n(t))_{t\in\mathbb{R}}$, where $t$ is the process parameter occurring in $h$. An example is given by the empirical kernel distribution $(H_n(t))_{t\in\mathbb{R}}$.

\vspace*{0.1cm}

\begin{defin}
	Let $h:\mathbb{R}^{m+1}\rightarrow\mathbb{R}$ be a measurable and bounded function, symmetric in the first $m$ arguments and non-decreasing in the last. Suppose that for all $x,y \in \mathbb{R}$ we have $\lim\limits_{t\rightarrow \infty}h(x,y,t)=1$ and $\lim\limits_{t\rightarrow -\infty}h(x,y,t)=0.$ 
	We call the process $(U_n(t))_{t\in \mathbb{R}}$ empirical $U$-distribution function. As $U$-distribution function we define  $U(t):=\mathbb{E}\left(h(Y_1,\ldots,Y_m,t)\right)$ for independent copies $Y_1,\ldots,Y_m$ of $X_1$. Then the empirical process is defined as 
	
	$$(\sqrt{n}(U_n(t)-U(t)))_{t \in \mathbb{R}}.$$
\end{defin}

Again, we use the Hoeffding decomposition in our proofs.
For fixed $t$ we have
\begin{align*}
U_n(t)=\frac{1}{\binom{n}{m}}\sum_{1\leq i_1<\ldots<i_m\leq n}h(X_{i_1},\ldots,X_{i_m},t)
\end{align*}
and therefore we can decompose $U_n(t)$  analogously to Definition \ref{hoeff}.

Additionally, the extended variation has to be transformed. \cite{Fischer.2016} have used the extended uniform variation condition, which has the same properties as the extended variation condition.

\vspace*{0.1cm}

Now we want to establish an invariance principle for the $U$-process. For near epoch dependent sequences on absolutely regular processes it has already been shown by \cite{Deh2002} and a result for strong mixing can be found in \cite{Wen2011}. Nevertheless, these results only consider the case of a bivariate kernel and therefore exclude such examples as given in Section 1.

\vspace*{0.1cm}
From now on we limit the considered $U$-process to $H_n$, that is $U_n(t)$ has the kernel $g(x_1,\ldots,x_m,t)=1_{[h(x_1,\ldots,x_m)\leq t]}$. Therefore, $U(t)=\mathbb{E}\left(1_{[h(Y_1,\ldots,Y_m)\leq t]}\right)=\mathbb{P}(h(Y_1,\ldots,Y_m)\leq t)=H_F(t)$ and since $H_F$ has density $h_F<\infty$, $H_F$ is Lipschitz-continuous.

\begin{theo}\label{invpri}
	$\newline$
	Let $h$ be a kernel with distribution function $H_F$ and related density $h_F<\infty$.
	Moreover let $g_1$ be the first term of the Hoeffding decomposition of $H_n$.
	Let $(X_n)_{n\in\mathbb{N}}$ be NED with approximation constants $(a_l)_{l\in\mathbb{N}}$ on an absolutely regular process $(Z_n)_{n\in\mathbb{Z}}$ with mixing coefficients $(\beta(l))_{l\in\mathbb{N}}$ with $\sum_{l=1}^\infty l^2 \beta^{\frac{\delta}{2+\delta}}(l)<\infty$ for a $0<\delta<1$. Moreover, let hold that $\sum_{l=1}^\infty l^2 a_l^{\frac{\delta}{2+2\delta}}<\infty$ . 
	Then
	\begin{align*}
	\left(\frac{m}{\sqrt{n}}\sum_{i=1}^n g_1(X_i,t)\right)_{t\in\mathbb{R}}\stackrel{D}{\longrightarrow}\left(W(t)\right)_{t\in \mathbb{R}},
	\end{align*}
	where $W$ is a Gaussian process having continuous path with probability 1.
\end{theo}

\begin{Rem}
	Note that the condition $\sum_{l=1}^\infty l^2 \left(L\sqrt{2A_l}\right)^{\frac{\delta}{1+\delta}}<\infty$ of \cite{Deh2002}, where $L$ is the variation constant of the kernel $g(x_1,\ldots,x_m)=1_{\left[h(x_1,\ldots,x_m)\leq t\right]}$ and $A_l=\sqrt{2\sum_{i=l}^\infty a_i}$, follows directly from the condition $\sum_{l=1}^\infty l^2 a_l^{\frac{\delta}{2+2\delta}}<\infty$ and is therefore omitted here.
		\end{Rem}

This Theorem can be proven in the same way as Theorem 4.8 of \cite{Deh2002} by choosing $g_t(x)=g_1(x,t)$ and $G(t)=H_F(t)$ (preserving the properties of the single functions) and is therefore omitted.

\vspace*{0.15cm}
\section{Main theorem}
Now we will state the main theorem of our paper, the asymptotic normality of $GL$-statistics of $L_1$-near epoch dependent processes on absolutely regular random variables. Under independence this result has been proven by \cite{Serf1984}, some of the lemmata can also be found in \cite{Cho1988}. Under strong mixing an analogous result can be found in \cite{Fischer.2016}.

\begin{theo}\label{main}
$\newline$
Let $h(x_1,\ldots,x_m)$ be a Lipschitz-continuous kernel with distribution function
$H_F$ and related density $0<h_F <\infty$ and for all $2\leq k \leq m$ and all $i_1<i_2<\ldots<i_m$ let $h_{F;X_{i_2},\ldots,X_{i_k}}$ be bounded. Moreover, let $J$ be a function with $J(t)=0$ for $t\notin\left[\alpha,\beta\right]$ , $0<\alpha<\beta<1$, and in $\left[\alpha,\beta\right]$ let $J$ be bounded and a.e. continuous concerning the Lebesgue-measure and a.e. continuous concerning $H_F^{-1}$.
Additionally, let  $X_1,\ldots,X_n$ be NED with approximation constants $(a_l)_{l\in\mathbb{N}}$ on an absolutely regular process $(Z_n)_{n\in\mathbb{Z}}$ with mixing coefficients $(\beta(l))_{l\in\mathbb{N}}$ with $\sum_{l=1}^\infty l^2 \beta^{\frac{\gamma}{2+\gamma}}(l)<\infty$ for a $0<\gamma<1$. Moreover, let be  $\sum_{l=1}^\infty l^2a_l^{\frac{\gamma}{2+2\gamma}}<\infty$ for a $\delta>8$.
Then for the $GL$-statistics $T(H_n)$ it holds that
\begin{align*}
\sqrt{n}\left(T(H_n)-T(H_F)\right)\stackrel{\mathcal{D}}{\longrightarrow}N(0,\sigma_{GL}^2),
\end{align*}
where 
\begin{align*}
\sigma_{GL}^2=&m^2\big(\var\left(\mathbb{E}\left(A(Y_1,\ldots,Y_{m})\vert Y_1=X_1\right)\right)\\
&+2\sum_{j=1}^\infty \cov\left( \mathbb{E}\left(A(Y_1,\ldots,Y_{m})\vert Y_1=X_1\right),\mathbb{E}\left(A(Y_1,\ldots,Y_{m})\vert Y_{1}=X_{j+1}\right)\right)\big)
\end{align*}
with independent copies $Y_1,\ldots,Y_m$ of $X_1$
and
\begin{align*}
A(x_1,\ldots,x_m)=&-\int_{-\infty}^{\infty}{\left(1_{\left[h(x_1,\ldots,x_m)\leq y\right]}-H_F(y)\right)J(H_F(y))dy}\\
&+\sum_{i=1}^{d}{a_i\frac{p_i-1_{\left[h(x_1,\ldots,x_m)\leq H_F^{-1}(p_i)\right]}}{h_F(H_F^{-1}(p_i))}}.
\end{align*}
\end{theo}

Analogous to the case of $U$-statistics, to use this theorem for example for confidence intervals in applications the problem arises how to handle the asymptotic variance $\sigma_{GL}^2$. Normally, this is unknown due to the unknown conditional expected values and the unknown distribution. Therefore, it is necessary to find an estimator for $\sigma^2_{GL}$. The following corollary closes this gap. It is stated for $L_1$-NED but it is also possible to show an analogous result under strong mixing such that it is applicable for the theorem of \cite{Fischer.2016}.

\begin{coro}\label{kern}
	Let $h:\mathbb{R}^m\rightarrow \mathbb{R}$ be a Lipschitz-continuous kernel.
	
	Moreover, let $(X_n)_{n\in\mathbb{N}}$ be NED with approximation constants $(a_l)_{l\in\mathbb{N}}$ on an absolutely regular process $(Z_n)_{n\in\mathbb{Z}}$ with mixing coefficients $(\beta(l))_{l\in\mathbb{N}}$ and let a $\delta>11$  exist, such that $\sum_{l=1}^\infty l\beta^{2/(2+\delta)(l)}<\infty$ and $a_l=O\left(l^{-\delta-3}\right)$. The weight function $\kappa$ and the bandwidth $b_n$ should fulfil Assumption \ref{asskern}. Then it holds that $\hat{\sigma}_{GL}^2\stackrel{\mathbb{P}}{\longrightarrow} \sigma_{GL}^2$ for $n\rightarrow \infty$ for the long-run variance estimator 
	\begin{align*}
	\hat{\sigma}_{GL}^2=\sum_{r=-(n-1)}^{n-1}\kappa\left(\frac{\vert r \vert }{b_n}\right)\frac{1}{n}\sum_{i=1}^{n-r}\hat{A}_1(X_i)\hat{A}_1(X_{i+r}),
	\end{align*}
	with 
	\begin{align*}
	\hat{A}_1(x)=\frac{1}{n^{m-1}}&\sum_{1\leq i_1<\ldots<i_{m-1}\leq n}A(x,X_{i_1},\ldots,X_{i_{m-1}})\\
	&-\frac{1}{n^m}\sum_{1\leq i_1<\ldots <i_m\leq n}A(X_{i_1},\ldots,X_{i_m})
	\end{align*}
	being the estimator for the first term of the Hoeffding-decomposition of $A$.
	
\end{coro}

The proofs of Theorem \ref{main} and Corollary \ref{kern} can be found in Section 5.

\section{Main proof}

In this section we will prove the missing results to finally assemble them to the main proof. Preliminary results needed in the single proofs can be found in the appendix.
\vspace*{0.1cm}

\begin{proof}[Proof of Theorem \ref{asynom}]
	$\newline$
	The proof makes use of the Hoeffding decomposition
	\begin{align*}
	\sqrt{n}(U_n-\theta)=\sqrt{n}\sum_{j=1}^m{\binom{m}{j}\frac{1}{\binom{n}{j}}S_{jn}}.
	\end{align*}
	We show that the linear part $\frac{m}{\sqrt{n}}\sum_{i=1}^{n}g_1(X_i)$ is asymptotically normal and that the remaining terms converge to $0$ in probability.
	$\newline$

	We know that $g_1$ is bounded because $h$ is bounded. Using Lemma 2.1.7 of \cite{Wen2011} we also know, that $g_1$ is NED with approximation constants $a'_l=Ca_l^{\frac{1}{2}}$. 
	
	Together with the above assumptions it is
	\begin{align*}
	\sum_{l=1}^\infty\beta(l)<\infty\qquad \text{and} \qquad \sum_{l=1}^{\infty}a'_l<\infty.
	\end{align*}
	
	We have made this considerations to finally apply Theorem 2.3 of \cite{Ibragimov.1961} getting
	
	\begin{align*}
	\frac{m}{\sqrt{n}}\sum_{i=1}^n g_1(X_1)\stackrel{D}{\longrightarrow}N(0,m^2\sigma^2)
	\end{align*}
	with $\sigma^2=\var(g_1(X_1))+2\sum_{j=1}^{\infty}\cov(g_1(X_1),g_1(X_{1+j}))<\infty$.
	
	For the remaining terms we want to use Lemma \ref{finalcov} with constants $c_{i_1,i_{k+1}}=1$, needing $\sum_{l=0}^nl\left(\beta(l)+A_l\right)=O(n^\gamma)$ for a $\gamma\geq 0$.
	
	Using
	\begin{align*}
	A_l=\left(2\sum_{i=l}^\infty a_i\right)^{\frac{1}{2}}\leq \left(2C\sum_{i=l}^\infty i^{-\delta-2}\right)^{\frac{1}{2}}=O\left(n^{-\frac{\delta+1}{2}}\right)
	\end{align*}
	we get
	\begin{align*}
	&\sum_{l=0}^nl\left(\beta(l)+A_l\right)\leq C \sum_{l=1}^nl\left(l^{-\delta}+n^{-\frac{\delta+1}{2}}\right)\leq C\sum_{l=1}^nl^{-\delta+1}+n^{-\frac{\delta+1}{2}}\sum_{l=1}^n l\\
	&= O(n^\gamma)+O\left(n^{2-\frac{\delta+1}{2}}\right)= O(n^\gamma)
	\end{align*}
	and so Lemma \ref{finalcov} is applicable and the convergence of the remaining terms is analogous to the strong mixing case in \cite{Fischer.2016} and therefore omitted.
	
\end{proof}

\begin{proof}[Proof of Theorem \ref{kernest}]
		$\newline$
	By decomposing the estimator into two parts we can apply the results of \cite{deJong.2000}:
	\begin{align*}
	\hat{\sigma}^2=&\sum_{r=-(n-1)}^{n-1}\kappa\left(\frac{\vert r\vert}{b_n}\right)\frac{1}{n}\sum_{j=1}^{n-\vert r \vert}\Big(\hat{g}_1(X_i)\hat{g}_1(X_{i+\vert r \vert})\\
	&\phantom{\sum_{r=-(n-1)}^{n-1}\kappa\left(\frac{\vert r\vert}{b_n}\right)\frac{1}{n}\sum_{j=1}^{n-\vert r \vert}}-g_1(X_i)g_1(X_{i+\vert r \vert})+h_1(X_i)h_1(X_{i+\vert r \vert})\Big)\\
	=&\sum_{r=-(n-1)}^{n-1}\kappa \left(\frac{\vert r \vert}{b_n}\right)\frac{1}{n}\sum_{j=1}^{n-\vert r \vert } g_1(X_i)g_1(X_{i+\vert r \vert})\\
	&+\sum_{r=-(n-1)}^{n-1}\kappa \left(\frac{\vert r \vert}{b_n}\right)\frac{1}{n}\sum_{j=1}^{n-\vert r \vert } \left(\hat{g}_1(X_i)\hat{g}_1(X_{i+\vert r \vert})-g_1(X_i)g_1(X_{i+\vert r\vert})\right)	
	\end{align*}
	
	For the first term \cite{deJong.2000} showed that it converges to $\sigma^2$ in probability. Therefore, it remains to show
	\begin{align*}
	\mathbb{E}\Big\vert\sum_{r=-(n-1)}^{n-1}\kappa \left(\frac{\vert r \vert}{b_n}\right)\frac{1}{n}\sum_{j=1}^{n-\vert r \vert } \left(\hat{g}_1(X_i)\hat{g}_1(X_{i+\vert r \vert})-g_1(X_i)g_1(X_{i+\vert r\vert})\right)\Big\vert\longrightarrow 0.
	\end{align*}
	
	Let us first expand $g_1(x)-\hat{g}_1(x)$ into single terms.
	
	\begin{align*}
	& g_1(x)-\hat{g}_1(x)\\
	=& g_1(x)-\frac{1}{n^{m-1}}\sum_{1\leq i_1<\ldots <i_{m-1}\leq n}h(x, X_{i_1},\ldots, X_{m-1})\\
	&\phantom{g_1(x)}+\frac{1}{n^m}\sum_{1\leq i_1<\ldots < i_m\leq n}h(X_{i_1},\ldots,X_{i_m})\\
	=& g_1(x)- \frac{1}{n^{m-1}}\sum_{1\leq i_1<\ldots <i_{m-1}\leq n}\Big(g_m(x,X_{i_1},\ldots,X_{i_{m-1}})+g_1(x)+\sum_{j=1}^{m-1}g_1(X_{i_j})\\
	&\phantom{g_1(x)- \frac{1}{n^{m-1}}\sum_{1\leq i_1<\ldots <i_{m-1}}}+\ldots+\sum_{1\leq j_1<\ldots<j_{m-2}\leq m-1}g_{m-1}(x,X_{i_{j_1}},\ldots,X_{i_{j_{m-2}}})\\
	&\phantom{g_1(x)- \frac{1}{n^{m-1}}\sum_{1\leq i_1<\ldots <i_{m-1}}} \sum_{1\leq j_1<\ldots<j_{m-1}\leq m-1}g_{m-1}(X_{i_{j_1}},\ldots,X_{i_{j_{m-1}}})\Big)\\
	&+ \frac{1}{n^m}\sum_{1\leq i_1<\ldots < i_m\leq n}\Big(g(X_{i_1},\ldots,X_{i_m})+\sum_{j=1}^m g(X_{i_j})\\
	&\phantom{\frac{1}{n^m}\sum_{1\leq i_1<\ldots < i_m\leq n}}	\sum_{1\leq j_1<\ldots<j_{m-1}\leq m}g_{m-1}(X_{i_{j_1}},\ldots,X_{i_{j_{m-1}}})\Big)\\
	=& -\frac{1}{n^{m-1}}\sum_{1\leq i_1<\ldots <i_{m-1}\leq n}g_m(x,X_{i_1},\ldots,X_{i_{m-1}})-(m-1)\frac{1}{n}\sum_{i=1}^ng_1(X_i)\\
	&-\ldots -2 \frac{1}{n^{m-2}}\sum_{1\leq i_1< \ldots <i_{m-2}\leq n}g_{m-1}(x,X_{i_1},\ldots,X_{i_{m-2}})\\
	&-\frac{1}{n^{m-1}}\sum_{1\leq i_1<\ldots <i_m\leq n}g_{m-1}(X_{i_1},\ldots,X_{i_{m-1}})\\
	&+\frac{1}{n^m}\sum_{1\leq i_1<\ldots <i_m\leq n}g_m(X_{i_1},\ldots,X_{i_m})+m\frac{1}{n}\sum_{i=1}^ng_1(X_i)\\
	&+\ldots + 2\frac{1}{n^{m-1}}\sum_{1\leq i_1<\ldots <i_{m-1}\leq n}g_{m-1}(X_{i_1},\ldots,X_{i_{m-1}})\\
	=& \frac{1}{n}\sum_{i=1}^ng_1(X_i)-(m-2)\frac{1}{n}\sum_{i=1}^ng_2(x,X_i)+(m-2)\frac{1}{n^2}\sum_{1\leq i<j\leq n}g_2(X_i,X_j)\\
	&-\ldots-2\frac{1}{n^{m-2}}\sum_{1\leq i_1<\ldots<i_{m-2}\leq n}g_{m-1}(x,X_{i_1},\ldots,X_{i_{m-2}})\\
	&+2\frac{1}{n^{m-1}}\sum_{1\leq i_1<\ldots <i_{m-1}\leq n}g_{m-1}(X_{i_1},\ldots,X_{i_{m-1}})\\
	&- \frac{1}{n^{m-1}}\sum_{1\leq i_1<\ldots<i_{m-1}\leq n}g_{m}(x,X_{i_1},\ldots,X_{i_{m-1}})+\frac{1}{n^{m}}\sum_{1\leq i_1<\ldots <i_{m}\leq n}g_{m-1}(X_{i_1},\ldots,X_{i_{m}})\\
	=&\frac{1}{n}\sum_{i=1}^ng_1(X_i)\\
	&-\sum_{k=2}^m (m-k)\frac{1}{n^{k-1}}\left(\sum_{1\leq i_1<\ldots<i_{k-1}\leq n}g_k(x,X_{i_1},\ldots, X_{i_{k-1}})+\frac{1}{n}\sum_{1\leq i_1<\ldots< i_k \leq n}g_k(X_{i_1},\ldots,X_{i_k})\right)
	\end{align*}
	
	Using this representation we can split the expected value and handle the single terms separately.
	
	\begin{align*}
	&\mathbb{E}\Big \vert \sum_{r=-(n-1)}^{n-1}\kappa\left(\frac{\vert r \vert }{b_n}\right) \frac{1}{n}\sum_{j=1}^{n-\vert r \vert } \left(g_1(X_j)g_1(X_{j+\vert r \vert })-\hat{g}_1(X_j)\hat{g}_1(X_{j+\vert r \vert })\right)\Big \vert\\
	&\leq \mathbb{E}\Big \vert \sum_{r=-(n-1)}^{n-1}\kappa\left(\frac{\vert r \vert }{b_n}\right) \frac{1}{n}\sum_{j=1}^{n-\vert r \vert } \left( (g_1(X_j)-\hat{g}_1(X_j))g_1(X_{j+\vert r \vert })\right)\Big \vert\\
	& + \mathbb{E}\Big \vert \sum_{r=-(n-1)}^{n-1}\kappa\left(\frac{\vert r \vert }{b_n}\right) \frac{1}{n}\sum_{j=1}^{n-\vert r \vert } \left((g_1(X_{j+\vert r \vert })-\hat{g}_1(X_{j+\vert r \vert }))\hat{g}_1(X_j)\right)\Big \vert\\
	&\leq \mathbb{E}\Big \vert \sum_{r=-(n-1)}^{n-1}\kappa\left(\frac{\vert r \vert }{b_n}\right) \frac{1}{n}\sum_{j=1}^{n-\vert r \vert } \frac{1}{n}\sum_{i=1}^ng_1(X_i)g_1(X_{j+\vert r \vert })\Big \vert\\
	&+ \mathbb{E}\Big \vert \sum_{r=-(n-1)}^{n-1}\kappa\left(\frac{\vert r \vert }{b_n}\right) \frac{1}{n}\sum_{j=1}^{n-\vert r \vert } (m-2)\frac{1}{n}\sum_{i=1}^n g_2(X_i,X_j)g_1(X_{j+\vert r \vert })\Big \vert\\
	&+\mathbb{E}\Big \vert \sum_{r=-(n-1)}^{n-1}\kappa\left(\frac{\vert r \vert }{b_n}\right) \frac{1}{n}\sum_{j=1}^{n-\vert r \vert }(m-2)\frac{1}{n^2}\sum_{1\leq i_1 <i_2\leq n}g_2(X_{i_1},X_{i_2})g_1(X_{j+\vert r \vert })\Big\vert\\
	&+\ldots \\
	&+\mathbb{E}\Big \vert \sum_{r=-(n-1)}^{n-1}\kappa\left(\frac{\vert r \vert }{b_n}\right) \frac{1}{n}\sum_{j=1}^{n-\vert r \vert }\frac{1}{n^{m-1}}\sum_{1\leq i_1< \ldots < i_{m-1}\leq n}g_m(X_j,X_{i_1},\ldots,X_{i_{m-1}})g_1(X_{j+\vert r \vert})\Big \vert\\
	&+\mathbb{E}\Big \vert \sum_{r=-(n-1)}^{n-1}\kappa\left(\frac{\vert r \vert }{b_n}\right) \frac{1}{n}\sum_{j=1}^{n-\vert r \vert }\frac{1}{n^m}\sum_{1\leq i_1<\ldots < i_m\leq n}g_m(X_{i_1},\ldots,X_{i_m})g_1(X_{j+\vert r \vert })\Big\vert\\
	&+\mathbb{E}\Big \vert \sum_{r=-(n-1)}^{n-1}\kappa\left(\frac{\vert r \vert }{b_n}\right) \frac{1}{n}\sum_{j=1}^{n-\vert r \vert } \frac{1}{n}\sum_{i=1}^ng_1(X_i)\hat{g}_1(X_{j})\Big \vert\\
	&+\ldots\\
	&+\mathbb{E}\Big \vert \sum_{r=-(n-1)}^{n-1}\kappa\left(\frac{\vert r \vert }{b_n}\right) \frac{1}{n}\sum_{j=1}^{n-\vert r \vert }\frac{1}{n^m}\sum_{1\leq i_1<\ldots < i_m\leq n}g_m(X_{i_1},\ldots,X_{i_m})\hat{g}_1(X_{j})\Big\vert
	\end{align*}
	We denote the $4(m-1)+2$ terms with $I_{i}$, $i=1,\ldots, 4(m-1)+1$.
	
	The first term $I_1$  containing the first term of the Hoeffding decomposition can be handled analogously to \cite{Dehling.2015}, using the Lemma C.1 of them. The conditions of the Lemma are fulfilled, since the boundedness of the kernel $h$ replaces the Assumption 2.3. It remains to show, that our definition of NED implies the required assumptions on the P-NED process.
	Therefore, we want to use Lemma A.1 of \cite{Dehling.2015} saying that an $L_1$-near epoch dependent process on $(Z_n)_{n\in \mathbb{Z}}$ with approximating constants $(a_l)_{l\in\mathbb{N}}$ is P-NED on the same process $(Z_n)_{n\in\mathbb{Z}}$. If we then choose $s_k=Ck^{-6(1+\frac{2+\delta}{\delta})}$ and $\Phi(x)=x^{-1}$ we gain 
	\begin{align*}
	\Phi(\epsilon)s_k=\epsilon^{-1}Ck^{-6(1+\frac{2+\delta}{\delta})}\geq k^{-(\delta+3)}\epsilon^{-1}=a_k\epsilon^{-1}
	\end{align*}
	and hence we know that $(X_n)_{n\in\mathbb{N}}$ is P-NED with approximating constants $(s_k)_{k\in\mathbb{N}}$ and function $\Phi$ on an absolutely regular process $(Z_n)_{n\in\mathbb{Z}}$ with mixing coefficients $(\beta(l))_{l\in\mathbb{N}}$. For these coefficients holds $s_k\Phi(k^{-6})=O(k^{-6(2+\delta)/\delta})$ and $\sum_{k=1}^\infty k\beta_k^{\delta/(2+\delta)}<\infty$ and so all assumptions needed for Lemma C.1 are fulfilled.
	
	Let us now consider all the terms, which contain $g_1$ and $g_k(X_j,\ldots)$, $k=1,m$. These are the terms $I_{2k}$, $k=1,\ldots,m-1$.
	
	\begin{align*}
	I_{2k}&= \mathbb{E}\Big \vert \sum_{r=-(n-1)}^{n-1}\kappa\left(\frac{\vert r \vert }{b_n}\right) \frac{1}{n}\sum_{j=1}^{n-\vert r \vert }(m-(k+1))\frac{1}{n^k}\sum_{1\leq i_1<\ldots < i_k\leq n}g_{k+1}(X_j,X_{i_1},\ldots,X_{i_k})h_1(X_{j+\vert r \vert })\Big\vert\\
	&=\mathbb{E}\Big\vert \frac{1}{n}\sum_{j_1=1}^n \frac{m-(k+1)}{n^k} \sum_{1\leq i_1< \ldots < i_{k+1}\leq n} g_{k+1}(X_{i_1},\ldots,X_{i_{k+1}})g_1(X_{j_1})\kappa\left(\frac{\vert i_1-j_1 \vert }{b_n}\right)\Big\vert\\
	&\leq  \mathbb{E}\Bigg(\left(\frac{1}{n}\sum_{j_1=1}^n g_1(X_{j_1})^2\right)^{\frac{1}{2}} \\
	&\phantom{\leq \mathbb{E}}\left( \frac{1}{n}\sum_{j_1=1}^n \left(\frac{m-(k+1)}{n^k}\sum_{1\leq i_1 <\ldots <i_{k+1}\leq n} g_{k+1}(X_{i_1},\ldots,X_{i_{k+1}})\kappa\left(\frac{\vert i_1-j_1 \vert }{b_n}\right)\right)^2\right)^\frac{1}{2}\Bigg)\\
	&\leq \left(\mathbb{E}\left(\frac{1}{n}\sum_{j_1=1}^ng_1(X_{j_1})^2\right)\right)^\frac{1}{2}\\
	&\phantom{\leq\mathbb{E}}\left(\mathbb{E}\left(\frac{1}{n}\sum_{j_1=1}^n \left(\frac{m-(k+1)}{n^k}\sum_{1\leq i_1 <\ldots <i_{k+1}\leq n} g_{k+1}(X_{i_1},\ldots,X_{i_{k+1}})\kappa\left(\frac{\vert i_1-j_1 \vert }{b_n}\right)\right)^2\right)\right)^\frac{1}{2},\\
	\end{align*}
	where we used the Hölder-inequality in the last step. Now we can use the boundedness of $h_1$ (since $h$ is bounded) and get
	\begin{align}
	I_{2k}&\leq \left(\mathbb{E}\left(\frac{1}{n}\sum_{j_1=1}^ng_1(X_{j_1})^2\right)\right)^\frac{1}{2} \nonumber\\
	&\phantom{\leq\mathbb{E}}\left(\mathbb{E}\left( \frac{(m-(k+1))^2}{n^{2k+1}}\left(\sum_{1\leq i_1 <\ldots <i_{k+1}\leq n} g_{k+1}(X_{i_1},\ldots,X_{i_{k+1}})\sum_{j_1=1}^n\kappa\left(\frac{\vert i_1-j_1 \vert }{b_n}\right)\right)^2\right)\right)^\frac{1}{2}\nonumber\\
	&\leq C \frac{m-(k+1)}{n^{k+\frac{1}{2}}} \left(\mathbb{E}\left(\sum_{1\leq i_1 <\ldots <i_{k+1}\leq n} g_{k+1}(X_{i_1},\ldots,X_{i_{k+1}})\sum_{j_1=1}^n\kappa\left(\frac{\vert i_1-j_1 \vert }{b_n}\right)\right)^2\right)^\frac{1}{2}\nonumber\\
	&\leq C^\prime \frac{1}{n^{k+\frac{1}{2}}}\Bigg(\mathbb{E}\Bigg(\sum_{i_1,\ldots,i_{2(k+1)}=1}^n g_{k+1}(X_{i_1},\ldots,X_{i_{k+1}})g_{k+1}(X_{i_{k+2}},\ldots,X_{i_{2(k+1)}}) \label{I2k}\\
	&\phantom{\leq C^\prime \frac{1}{n^{k+\frac{1}{2}}}(\mathbb{E}(\sum_{i_1,\ldots,i_{2(k+1)}=1}^n} \sum_{j_1=1}^n\kappa\left(\frac{\vert i_1-j_1 \vert }{b_n}\right)\kappa\left(\frac{\vert i_{k+2}-j_1 \vert }{b_n}\right)\Bigg)\Bigg)^\frac{1}{2}\nonumber.
	\end{align}
	To show the convergence of $I_{2k}$ to zero we finally want to apply Lemma \ref{finalcov}. For this we have to show that 
	\begin{align*}
	\sum_{l=0}^nl(\beta(l)+A_l)=O(n^\gamma)
	\end{align*}
	with $A_l=\sqrt{2\sum_{i=l}^\infty a_i}$.
	From the assumption $\sum_{l=1}^\infty l \beta^{\delta/(2+\delta)}(l)<\infty$, which implies $l\beta^{\delta/(2+\delta)}(l)\rightarrow 0$ for $l\rightarrow \infty$, and the fact that the mixing coefficients $(\beta(l))_{l\in\mathbb{N}}$ are non-negative and monotone decreasing (therefore $\beta^{\delta/(2+\delta)}(l)$ is also monotone decreasing) we know, that $l\beta^{\delta/(2+\delta)}(l)$ is monotone decreasing and positive. Hence 
	\begin{align*}
	l\beta^{\delta/(2+\delta)(l)}=O\left(\frac{1}{l}\right)
	\end{align*}
	and so $\beta(l)=O\left(l^{-\eta}\right)$ for a $\eta>2$.
	Analogously to the proof of Theorem \ref{asynom}, but now for $a_l=O\left(l^{-\delta-3}\right)$, we then can show
	\begin{align*}
	\sum_{l=0}^n l(\beta(l)+A_l)\leq C\sum_{l=1}^n l^{-\eta+1}+O(n^{2-\frac{\delta+2}{2}}),
	\end{align*}
	where $2-\frac{\delta+2}{2}<\frac{1}{2}$ since $\delta>11$. Let us now have a closer look at the first term. We want to show that $C\sum_{l=1}^n l^{-\eta+1}=O(n^\gamma)$ for a $0<\gamma<1/2$. It is
	\begin{align*}
	\frac{\sum_{l=1}^nl^{-\eta+1}}{n^\gamma}\leq \inf\limits_{1\leq t \leq n}t^{-\gamma}\sum_{l=1}^nl^{1-\eta}\leq \sum_{l=1}^nl^{-1-\eta}l^{-\gamma}=\sum_{l=1}^n\left(\frac{1}{l}\right)^{-1+\eta+\gamma}.
	\end{align*}
	This is the Dirichlet series and it converges for $-1+\eta+\gamma>1$. Since $\eta>2$ and $0<\gamma<1/2$ we have
	$
	\sum_{l=1}^nl^{-\eta+1}=O(n^\gamma)
	$
	and therefore 
	\begin{align*}
	\sum_{l=0}^nl(\beta(l)+A_l)=O(n^\gamma)
	\end{align*}
	for a $0<\gamma<1/2$.
	
	Now we can apply Lemma \ref{finalcov} to (\ref{I2k}), where $$c_{i_1,i_{k+2}}=\sum_{j_1=1}^n\kappa\left(\frac{\vert i_1-j_1\vert}{b_n}\right)\kappa\left(\frac{\vert i_{k+2}-j_1\vert}{b_n}\right)=O(b_n)$$  and obtain
	\begin{align*}
	I_{2k}\leq C \frac{1}{n^{k+\frac{1}{2}}}\left(n^{2(k+1)-2+\gamma}b_n\right)^\frac{1}{2}\leq C\left(\frac{n^{2k+\gamma}}{n^{2k+1}}b_n\right)^\frac{1}{2}\\
	=C \left(n^{\gamma-\frac{1}{2}}\right)^\frac{1}{2}\left(\frac{b_n}{\sqrt{n}}\right)^\frac{1}{2}\longrightarrow 0,\\
	\end{align*}
	because of Assumption \ref{asskern} and $0<\gamma<1/2$.
	
	Therefore, $I_{2k}$ converges to zero for all $k=1,\ldots,m-1$. 
	
	\vspace{0.3cm}
	
	It remains to show the convergence of the remaining terms. The terms containing $g_1(\cdot)$ and $g_k(X_{i_1},\ldots,X_{i_k})$, $k=2,\ldots,m$ are denoted with $I_{2k+1}$, $k=1,\ldots,m-1$.
	
	\begin{align*}
	I_{2k+1}&=\mathbb{E}\Bigg\vert\sum_{r=-(n-1)}^{n-1}\kappa\left(\frac{\vert r \vert }{b_n}\right) \frac{1}{n}\sum_{j=1}^{n-\vert r \vert}\frac{m-(k+1)}{n^{k+1}}\sum_{1\leq i_1<\ldots<i_{k+1}\leq n}g_{k+1}(X_{i_1},\ldots,X_{i_{k+1}})g_1(X_{j+\vert r \vert })\Bigg\vert\\
	&\leq \Vert \frac{m-(k+1)}{n^{k+1}}\sum_{1\leq i_1<\ldots<i_{k+1}\leq n}g_{k+1}(X_{i_1},\ldots,X_{i_{k+1}})\Vert_2\\
	&\phantom{\leq \Vert} \Vert \frac{1}{n} \sum_{j=1}^{n-\vert r \vert}g_1(X_{j+\vert r \vert})\sum_{j_1=1}^n \kappa\left(\frac{\vert j-j_1 \vert }{b_n}\right)\Vert_2\\
	&\leq C\frac{1}{n^{k+1}}\left(n^{2(k+1)-\gamma}\right)^\frac{1}{2} \frac{b_n}{\sqrt{n}}\\
	&=C\frac{n^{k+\gamma/2}}{n^{k+1}}o(1)\longrightarrow 0,	
	\end{align*}
	where we used the H\"older inequality, the boundedness of $h_{k+1}$, Lemma \ref{finalcov} and Assumption \ref{asskern} as before.
	
	The convergence of the remaining terms $I_{2m}$, $I_{2k}$ and $I_{2k+1}$ for $k=m+1,\ldots,2(m-1)$ can be shown analogously and is therefore omitted.
\end{proof}

\begin{proof}[Proof of Theorem \ref{main}]
		$\newline$
	Analogously to the three conditions of \cite{Serf1984}, which were also used in \cite{Fischer.2016}, we want to show
	
	\begin{enumerate}[(i)]
		\item 
		For $W_{H_n,H_F}(y)=\left(\frac{\int_0^{H_n(y)}{J(t)dt}-\int_0^{H_F(y)}{J(t)dt}}{H_n(y)-H_F(y)}-J(H_F(y))\right)$ holds

		$\lVert W_{H_n,H_F} \rVert_{L_1}=o_p(1)$ and it is $\lVert H_n-H_F \rVert_{\infty}=O_p(n^{-\frac{1}{2}})$.
		\item  For the remainder term $R_{p_i,n}=\hat{\xi}_{p_i,n}-\xi_{p_i}+\frac{p_i-H_n(\xi_{p_i})}{h_f(\xi_{p_i})}$ of the Bahadur representation of an empirical quantile holds 
		\begin{align*}
		R_{p_i,n}=o_p(n^{-\frac{1}{2}}).
		\end{align*}
		\item  For a $U$-statistic with kernel 
		\begin{align*}
		A(x_1,\ldots,x_m)=&-\int_{-\infty}^{\infty}{\left(1_{\left[h(x_1,\ldots,x_m)\leq y\right]}-H_F(y)\right)J(H_F(y))dy}\\
		&+\sum_{i=1}^{d}{a_i\frac{p_i-1_{\left[h(x_1,\ldots,x_m)\leq H_F^{-1}(p_i)\right]}}{h_F(H_F^{-1}(p_i))}}
		\end{align*}
		we have
		\begin{align*}
		\sqrt{n}(U_n(A)-\theta)\stackrel{D}{\longrightarrow}N(0,\sigma^2).
		\end{align*}
	\end{enumerate} 
	
	As in the case of strong mixing condition (i) is fulfilled using Lemma 8.2.4.A of \cite{serf1980} and Corollary \ref{glican}. Condition (ii) can be proven by Lemma \ref{bahad}.

	\vspace*{0.15cm}
	
	It remains to show that condition (iii) is satisfied. 
	
	Therefore we apply Theorem \ref{asynom}. \cite{Fischer.2016} already showed that the kernel $A$ satisfies the assumptions in the theorem.
	
\end{proof}

\begin{proof}[Proof of Corollary \ref{kern}]
		$\newline$
	As we have mentioned before, the error term $T(H_n)-T(H_F)$ can be approximated by a $U$-statistic with kernel $A$. This leads to the special structure of $\sigma^2_{GL}$, being similar to that of a long-run variance of $U$-statistics. Therefore, we want to apply Theorem \ref{kernest} to prove this Corollary.
	It only remains to show the assumptions on the kernel $A$. As we have said above and was proved in \cite{Fischer.2016} the kernel $A$ is bounded and satisfies the extended variation condition. We therefore only have to show that $A$ also satisfies the $L_2$-variation condition.
	
	It is
	\begin{align*}
	&\sqrt{\mathbb{E}\left(\sup\limits_{\Vert(x_1,\ldots,x_m)-(X_1^\prime,\ldots,X_m^\prime)\Vert\leq \epsilon}\vert A(x_1,\ldots,x_m)-A(X_1^\prime,\ldots,X_m^\prime)\vert\right)^2}\\
	&=\Bigg(\mathbb{E}\Big(\sup\limits_{\Vert(x_1,\ldots,x_m)-(X_1^\prime,\ldots,X_m^\prime)\vert\leq \epsilon}\Bigg\vert-\int_{-\infty}^\infty(1_{[h(x_1,\ldots,x_m)\leq y]}-H_F(y))J(H_F(y))dy\\
	&\phantom{\mathbb{E}\sup\limits_{\Vert(x_1,\ldots,x_m)\vert\leq \epsilon}}+\int_{-\infty}^\infty(1_{[h(X^\prime_1,\ldots,X^\prime_m)\leq y]}-H_F(y))J(H_F(y))dy\\
	&\phantom{\mathbb{E}\sup\limits_{\Vert(x_1,\ldots,x_m)\vert\leq \epsilon}}+\sum_{i=1}^da_i\frac{p_i-1_{[h(x_1,\ldots,x_m)\leq H_F^{-1}(p_i)]}}{h_F(H_F^{-1}(p_i))}-\sum_{i=1}^da_i\frac{p_i-1_{[h(X^\prime_1,\ldots,X^\prime_m)\leq H_F^{-1}(p_i)]}}{h_F(H_F^{-1}(p_i))}\Bigg\vert\Big)^2\Bigg)^\frac{1}{2}\\
	&\leq \sqrt{\mathbb{E}\left(\sup\limits_{\Vert(x_1,\ldots,x_m)-(X_1^\prime,\ldots,X_m^\prime)\Vert\leq \epsilon}\Bigg\vert\int_{-\infty}^\infty(1_{[h(x_1,\ldots,x_m)\leq y]}-1_{[h(X^\prime_1,\ldots,X^\prime_m)\leq y]})J(H_F(y))dy\Bigg\vert\right)^2}\\
	&+ \sqrt{\mathbb{E}\left(\sup\limits_{\Vert(x_1,\ldots,x_m)-(X_1^\prime,\ldots,X_m^\prime)\Vert\leq \epsilon}\Bigg\vert\sum_{i=1}^da_i\frac{1_{[h(x_1,\ldots,x_m)\leq H_F^{-1}(p_i)]}-1_{[h(X^\prime_1,\ldots,X^\prime_m)\leq H_F^{-1}(p_i)]}}{h_F(H_F^{-1}(p_i))}\right)^2}
	\end{align*}
	These term can now be treated separately and analogous to the proof of the extended variation condition in \cite{Fischer.2016}. For the first term we gain
	\begin{align*}
	&\mathbb{E}\left(\sup\limits_{\Vert(x_1,\ldots,x_m)-(X_1^\prime,\ldots,X_m^\prime)\Vert\leq \epsilon}\Bigg\vert\int_{-\infty}^\infty(1_{[h(x_1,\ldots,x_m)\leq y]}-1_{[h(X^\prime_1,\ldots,X^\prime_m)\leq y]})J(H_F(y))dy\Bigg\vert\right)^2\\
	&\leq \mathbb{E}\left(\sup\limits_{\Vert(x_1,\ldots,x_m)-(X_1^\prime,\ldots,X_m^\prime)\Vert\leq \epsilon}\vert 1_{[h(X_1^\prime,\ldots,X_m^\prime)\in(t-\tilde{L}\epsilon,t+\tilde{L}\epsilon)]}\vert \Bigg\vert\int_{-\infty}^\infty J(H_F(y))dy\Bigg\vert\right)^2\\
    &\leq \mathbb{E}\left(\sup\limits_{t\in\mathbb{R}}\vert 1_{[h(X_1^\prime,\ldots,X_m^\prime)\in(t-\tilde{L}\epsilon,t+\tilde{L}\epsilon)]}\vert C\right)^2\\
    &\leq C\sup\limits_{t\in\mathbb{R}}\vert \mathbb{E}(1_{[h(X_1^\prime,\ldots,X_m^\prime)\in(t-\tilde{L}\epsilon,t+\tilde{L}\epsilon)]})^2\vert\\
    &\leq C\sup\limits_{t\in\mathbb{R}}\vert \mathbb{P}(h(X_1^\prime,\ldots,X_m^\prime)\in(t-\tilde{L}\epsilon,t+\tilde{L}\epsilon))\vert\\
    &\leq C(\sup\limits_{x\in\mathbb{R}}h_F(x))2\tilde{L}\epsilon\leq L\epsilon
	\end{align*}
	using the boundedness of $\vert \int_{-\infty}^\infty J(H_F(y))dy\vert$ and $h_F$.
	
	Therefore, 
	\begin{align*}
	\sqrt{\mathbb{E}\left(\sup\limits_{\Vert(x_1,\ldots,x_m)-(X_1^\prime,\ldots,X_m^\prime)\Vert\leq \epsilon}\vert A(x_1,\ldots,x_m)-A(X_1^\prime,\ldots,X_m^\prime)\vert\right)^2}\leq 2\sqrt{L\epsilon}
	\end{align*}
	and
		\begin{align*}
		\mathbb{E}\left(\sup\limits_{\Vert(x_1,\ldots,x_m)-(X_1^\prime,\ldots,X_m^\prime)\Vert\leq \epsilon}\vert A(x_1,\ldots,x_m)-A(X_1^\prime,\ldots,X_m^\prime)\vert\right)^2\leq L^\prime\epsilon
		\end{align*}
	This makes Theorem \ref{kernest} applicable and the proof is completed.
\end{proof}

\section{Application: Scale estimators under EGARCH-processes}
When developing theory for near epoch dependent data one of the widely used examples is the GARCH(p,q)-process (Generalized Autoregressive Conditional Heteroscedasticity) (\cite{Bollerslev.1986}), a generalisation of ARCH-processes. A process $(X_t)_{t\in\mathbb{Z}}$ is called GARCH(p,q)-process, if
\begin{align*}
X_t&=\sigma_tZ_t,
\end{align*}
with positive $\sigma^2$ given by
\begin{align*}
\sigma_t^2&=\alpha_0+\alpha_1X_{t-1}^2+\ldots+\alpha_pX_{t-p}^2+\beta_1\sigma_{t-1}^2+\ldots+\beta_q\sigma_{t-q}^2,
\end{align*}
where $\alpha_0,\ldots,\alpha_p,\beta_1,\ldots,\beta_q \in\mathbb{R}$ are non-negative with $\alpha_p\neq0$ and $\beta_q\neq 0$ and $(Z_t)_{t\in\mathbb{Z}}$ is an i.i.d. sequence with mean zero and variance equal to one.

\vspace{0.3cm}

\cite{Hansen.1991} relax the assumptions on $(Z_t)_{t\in\mathbb{Z}}$, such that $(Z_t)_{t\in\mathbb{Z}}$ can be assumed to be $\alpha$-mixing.
They showed then that if $\left(\mathbb{E}[(\beta_1+\alpha_1(\frac{X_t}{\sigma_t})^2)^r\vert \mathcal{F}_{t-1}]\right)^{1/5}\leq c <1 \text{ a.s. for all } t$, a GARCH(1,1)-process $X_t$ is $L_r$-NED on the $\alpha$-mixing process $Z_t\sigma_t$ with approximation constants $a_l=c^l2\alpha_0c/(1-c)$.

A generalisation which is widely used in financial applications and also in hydrology is the Exponential GARCH (EGARCH) model proposed by \cite{Nelson.1991}. One of the advantages of EGARCH-processes is that they do not have the non-negativity restriction of the GARCH-processes.
	
	The process $(X_t)_{t\in\mathbb{Z}}$ is called EGARCH(p,q)-process on the sequence $(Z_t)_{t\in\mathbb{Z}}$ , if
	\begin{align*}
	X_t&=\sigma_tZ_t
	\end{align*}
	where $\sigma_t^2$ is the positive conditional variance given by
	\begin{align*}
	\log(\sigma_t^2)&=\alpha_0+\alpha_1f(Z_{t-1})+\ldots+\alpha_pf(Z_{t-p})+\beta_1\log(\sigma_{t-1}^2)+\ldots+\beta_q\log(\sigma_{t-q}^2),
	\end{align*}
	where $\alpha_0,\ldots,\alpha_p,\beta_1,\ldots,\beta_q \in\mathbb{R}$ with $\alpha_p\neq0$ and $\beta_q\neq 0$ and $f$ is a measurable function which is linear in $Z$ with coefficients $\theta$ and $\lambda$ given by
	\begin{align*}
	f(Z_t)=\theta Z_t+\lambda(\vert Z_t\vert-\mathbb{E}\vert Z_t\vert).
	\end{align*}

In this section we want to show for some scale estimators proposed in Section 1 their asymptotic normality if we consider an underlying EGARCH(1,1)-process. Therefore, we first have to show under which conditions EGARCH(p,q)-processes are NED.

\begin{theo}\label{EGARCH}
	Let $\sigma_1$ be bounded and $\vert\sum_{i=1}^q\beta_i \vert<1$. Moreover, assume that
	\begin{align}\label{NEDEGARCH}
	\sup\limits_{t \in\mathbb{Z}}\vert Z_t\vert <\infty.
	\end{align}
	Then the EGARCH(p,q)-process on the sequence $(Z_t)_{t\in\mathbb{Z}}$ given by $X_t=\sigma_tZ_t$ is near epoch dependent.
\end{theo}
\begin{Rem}
	\begin{enumerate}
		\item The assumption (\ref{NEDEGARCH}) in Theorem \ref{EGARCH} is an analogue to the condition of \cite{Hansen.1991} for GARCH-processes to the EGARCH-case with arbitrary values $p$ and $q$.
		
		Whether this condition is fulfilled depends on the existing moments of $Z_i$. For example, if $(Z_t)_{t\in \mathbb{Z}}$ is a White Noise process with variance $\sigma$ (that is $\mathbb{E}\vert Z_t\vert\leq\sigma=1$), the condition is fulfilled.
		\item The boundedness of the conditional variance $\sigma_1$ is a common assumption for GARCH-processes (see \cite{Hansen.1991}, \cite{Lee.1994}). It results from the moment condition on $\sigma$, $\mathbb{E}\vert \sigma_1\vert^{1+\delta}<\infty$, which is needed in the following proof and the Lipschitz-condition.
	\end{enumerate}
\end{Rem}

\begin{proof}(Theorem \ref{EGARCH})
	
	Using an iterative expression of the term $\log(\sigma_t^2)$ we obtain
	\begin{align*}
	\log(\sigma_{t}^{2})=
	\sum_{j=1}^{n}\sum_{\substack{k_{1},\ldots,k_{q}\in \mathbb{N}_0,\\ k_{1}+\ldots+k_{q}=j-1}}\binom{j-1}{k_{1},\ldots,k_{q}}\beta_{1}^{k_{1}}\cdot\ldots\cdot\beta_{q}^{k_{q}} \left(\alpha_0+\sum^{p}_{k=1}\alpha_{k}f\left(Z_{t-k-\left(\sum_{i=1}^{q}i k_{i}\right)}\right)\right)\\
	+\sum_{\substack{k_{1},\ldots,k_{q}\in \mathbb{N}_{0},\\ k_{1}+\ldots+k_{q}=n}} \binom{n}{k_{1},\ldots,k_{q}} \beta_{1}^{k_{1}}\cdot\ldots\cdot\beta_{q}^{k_{q}}\log\left(\sigma_{t-\left(\sum_{i=1}^{q}i k_{i}\right)}^{2}\right).
	\end{align*}
	Now, considering the limit for $n\rightarrow \infty$, it is	
	\begin{align*}
	&\log(\sigma_{t}^{2})\\
	=& \lim_{n\rightarrow \infty} 
	\sum_{j=1}^{n}\sum_{\substack{k_{1},\ldots,k_{q}\in \mathbb{N}_{0},\\ k_{1}+...+k_{q}=j-1}}\binom{j-1}{k_{1},\ldots,k_{q}}\beta_{1}^{k_{1}}\cdot\ldots\cdot\beta_{q}^{k_{q}} \left(\alpha_0+\sum^{p}_{k=1}\alpha_{k}f\left(Z_{t-k-\left(\sum_{i=1}^{q}i k_{i}\right)}\right)\right) \\
	& +
	\lim_{n\rightarrow \infty} 
	\sum_{\substack{k_{1},\ldots,k_{q}\in \mathbb{N}_{0},\\ k_{1}+\ldots+k_{q}=n}} \binom{n}{k_{1},...,k_{q}} \beta_{1}^{k_{1}}\cdot\ldots\cdot\beta_{q}^{k_{q}}\log\left(\sigma_{t-\left(\sum_{i=1}^{q}i k_{i}\right)}^{2}\right)\\
	\end{align*}
	We want to show that the first term of the sum converges a.s. This is gained by the assumptions
	
	$\sup\limits_{t}\mathbb{E}\vert Z_t \vert<\infty$ and $\sum_{i=1}^q \beta_i\vert <1$ and the linearity of the function $f$. With the Multinomial Theorem and the convergence of the geometric series we can apply the monotone convergence theorem to obtain the  convergence of the series  (see for example Proposition 3.1.1 of \cite{Brockwell.2006}).
	
	For the second term we show that it converges to zero a.s., that is
	\begin{align*}
	\lim_{n\rightarrow \infty} 
	\sum_{\substack{k_{1},\ldots,k_{q}\in \mathbb{N}_{0},\\ k_{1}+\ldots+k_{q}=n}} \binom{n}{k_{1},\ldots,k_{q}} \beta_{1}^{k_{1}}\cdot\ldots\cdot\beta_{q}^{k_{q}}\log\left(\sigma_{t-\left(\sum_{i=1}^{q}i k_{i}\right)}^{2}\right)=0 \ ~ \text{for all } t \in \mathbb{Z} \text{  a.s.}
	\end{align*}
	By using the Multinomial Theorem we have
	\begin{align*}
	& \lim_{n\rightarrow \infty} 
	\sum_{\substack{k_{1},\ldots,k_{q}\in \mathbb{N}_{0},\\ k_{1}+\ldots+k_{q}=n}} \binom{n}{k_{1},\ldots,k_{q}} \beta_{1}^{k_{1}}\cdot\ldots\cdot\beta_{q}^{k_{q}}\log\left(\sigma_{t-\left(\sum_{i=1}^{q}i k_{i}\right)}^{2}\right)\\
	&\leq
	\sup_{k_{1},\ldots,k_{q}}\log\left(\sigma_{t-\left(\sum_{i=1}^{q}i k_{i}\right)}^{2}\right) \lim_{n\rightarrow \infty} 
	\sum_{\substack{k_{1},\ldots,k_{q}\in \mathbb{N}_{0},\\ k_{1}+\ldots+k_{q}=n}} \binom{n}{k_{1},\ldots,k_{q}} \beta_{1}^{k_{1}}\cdot\ldots\cdot\beta_{q}^{k_{q}}\\
	&=
	\sup_{k_{1},...,k_{q}}\log\left(\sigma_{t-\left(\sum_{i=1}^{q}i k_{i}\right)}^{2}\right) \lim_{n\rightarrow \infty} 
	(\beta_{1}+...+\beta_{q})^n
	\end{align*}
	and therefore the term converges a.s. to zero if 
	\begin{align*}
	\left|\sum_{i=1}^{q}\beta_{i}\right|<1.
	\end{align*}
	Hence, we can write
	\begin{align*}
	\log(\sigma_{t}^{2})=\sum_{j=1}^{\infty}\sum_{\substack{k_{1},...,k_{q}\in \mathbb{N}_{0},\\ k_{1}+...+k_{q}=j-1}}\binom{j-1}{k_{1},...,k_{q}}\beta_{1}^{k_{1}}\cdot...\cdot\beta_{q}^{k_{q}} \left(\alpha_0+\sum^{p}_{k=1}\alpha_{k}f\left(Z_{t-k-\left(\sum_{i=1}^{q}i k_{i}\right)}\right)\right).
	\end{align*}
	This is a linear solution and for this reason the process $(\log(\sigma_t^2))_t$ is near epoch dependent. Moreover, 
	\begin{align*}
	\sigma_t=\sqrt{\exp(\log(\sigma_t^2))}=g(\log(\sigma_t^2))
	\end{align*}
	with $g(x)=\sqrt{\exp(x)}$. This function $g$ fulfils the Lipschitz-condition for all $x\in (-\infty,a]$, $a\in \mathbb{R}$.
	We can now apply Proposition 2.11 of \cite{Boro}, where we need that $\sigma_1$ is bounded. Therefore, the process $\sigma_t$ and hence $X_t=\sigma_tZ_t$ is near epoch dependent on the process $(Z_t)_{t\in \mathbb{Z}}$.
\end{proof}
	
For the following simulations we consider an EGARCH(1,1)-process with parameters $\alpha=0.2$ and $\beta=0.05$. The coefficients of the function $f$ are chosen as $\theta=0.9$ and $\lambda=0.1$. These are common choices when simulating from an EGARCH-process. What is special and corresponds to the case of a NED-process on an underlying absolutely regular process is the choice of $Z_t$ as AR(1)-process with correlation coefficient $\rho=0.8$. Notice that the assumptions of Theorem \ref{NEDEGARCH} are therefore fulfilled. We compare the three estimators of Section 1 with different sample length concerning their asymptotic normality using QQ-Plots. For the estimators in Example 1.2 and 1.3 we use the special cases $Q$ and $LMS_n$ of the estimators.	

	\begin{figure}
		\caption{Gini's Mean difference for $n=100$ (left) and $n=1000$ (right)}
		\includegraphics[width=0.475\textwidth, height=0.4\textheight]{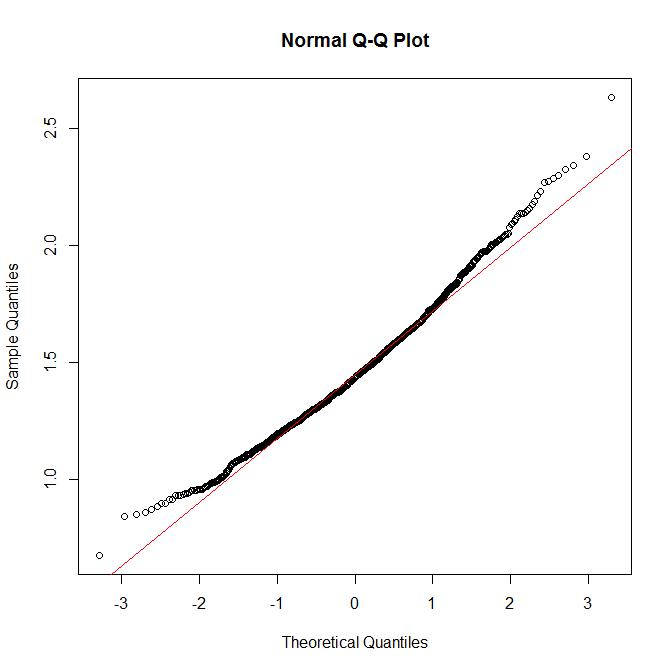}
		\hfill
		\includegraphics[width=0.475\textwidth, height=0.4\textheight]{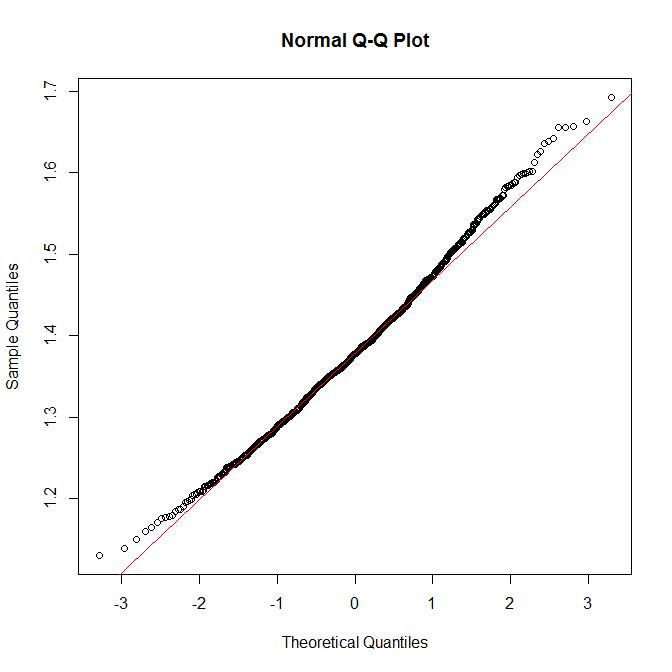}
	\end{figure}

	\begin{figure}
		\caption{$LMS_n$-estimator for $n=100$ (left) and $n=1000$ (right)}
		\includegraphics[width=0.475\textwidth, height=0.4\textheight]{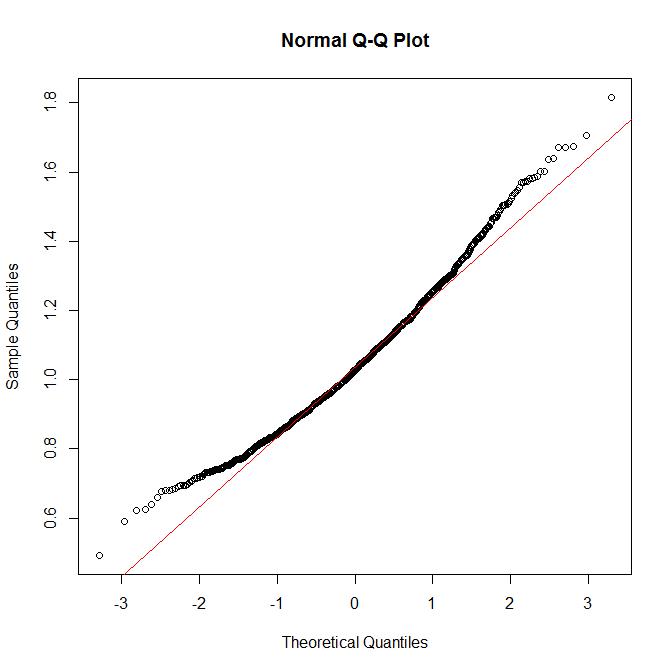}
		\hfill
		\includegraphics[width=0.475\textwidth, height=0.4\textheight]{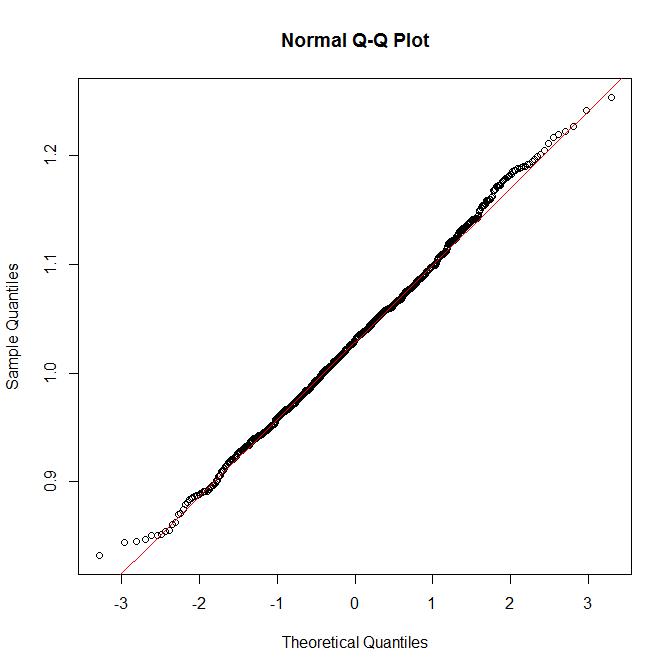}
	\end{figure}

	\begin{figure}
		\caption{Q-estimator for $n=100$ (left) and $n=1000$ (right)}
		\includegraphics[width=0.475\textwidth, height=0.4\textheight]{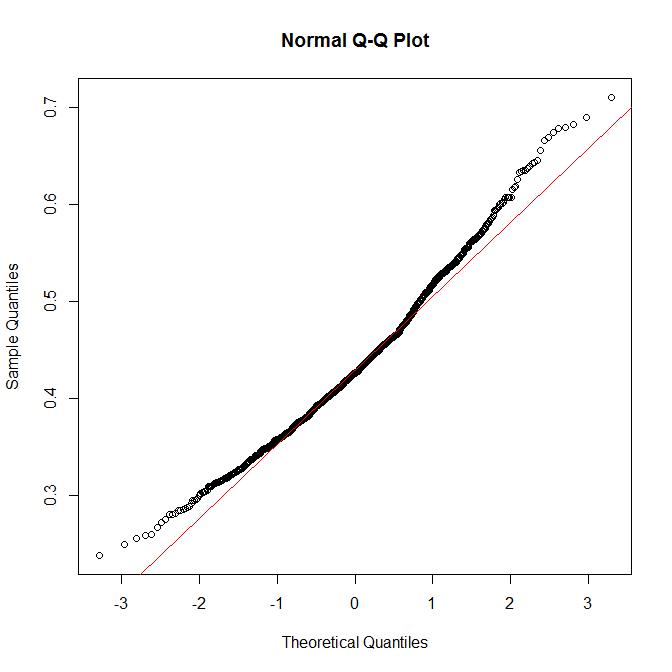}
		\hfill
		\includegraphics[width=0.475\textwidth, height=0.4\textheight]{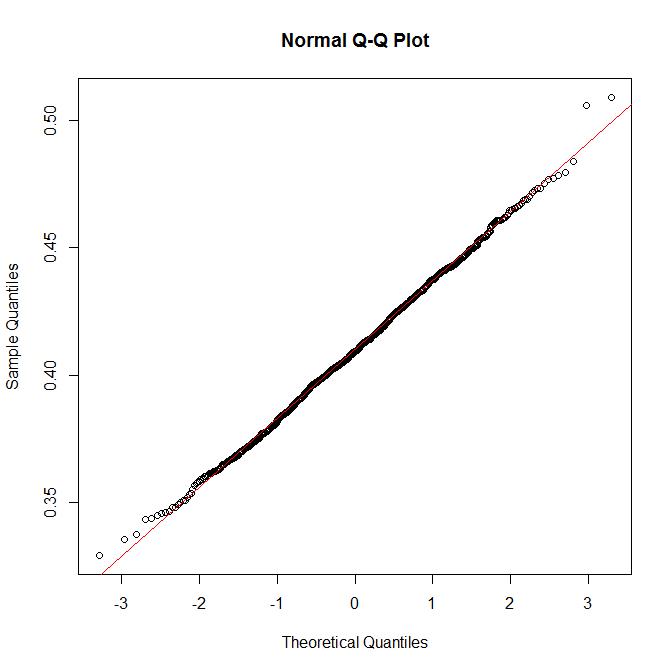}
	\end{figure}

\vspace{0.3cm}

The asymptotic normality of these estimators is therefore confirmed, although a sample length of about $n=1000$ is needed. The Gini's mean difference estimator proves to need the largest sample of the three estimators to be approximated well by a normal distribution.

It a second scenario we want to increase the dependence in the EGARCH process such that the AR(1) process as well as the EGARCH-process show very strong dependence. For this we choose $\alpha=0.8$ and $\beta=0.1$. The results can be found in Fig. \ref{Gmd2}-\ref{Qest2}. Because of the increased dependency within the EGARCH-process a larger sample size is needed to obtain a good approximation by the normal distribution. Similar to the first scenario, Gini's mean difference needs many data for a good approximation, whereas the $Q$-estimator does not need many more data in the presence of strong dependence.

\begin{figure}
	\includegraphics[width=0.475\textwidth, height=0.4\textheight]{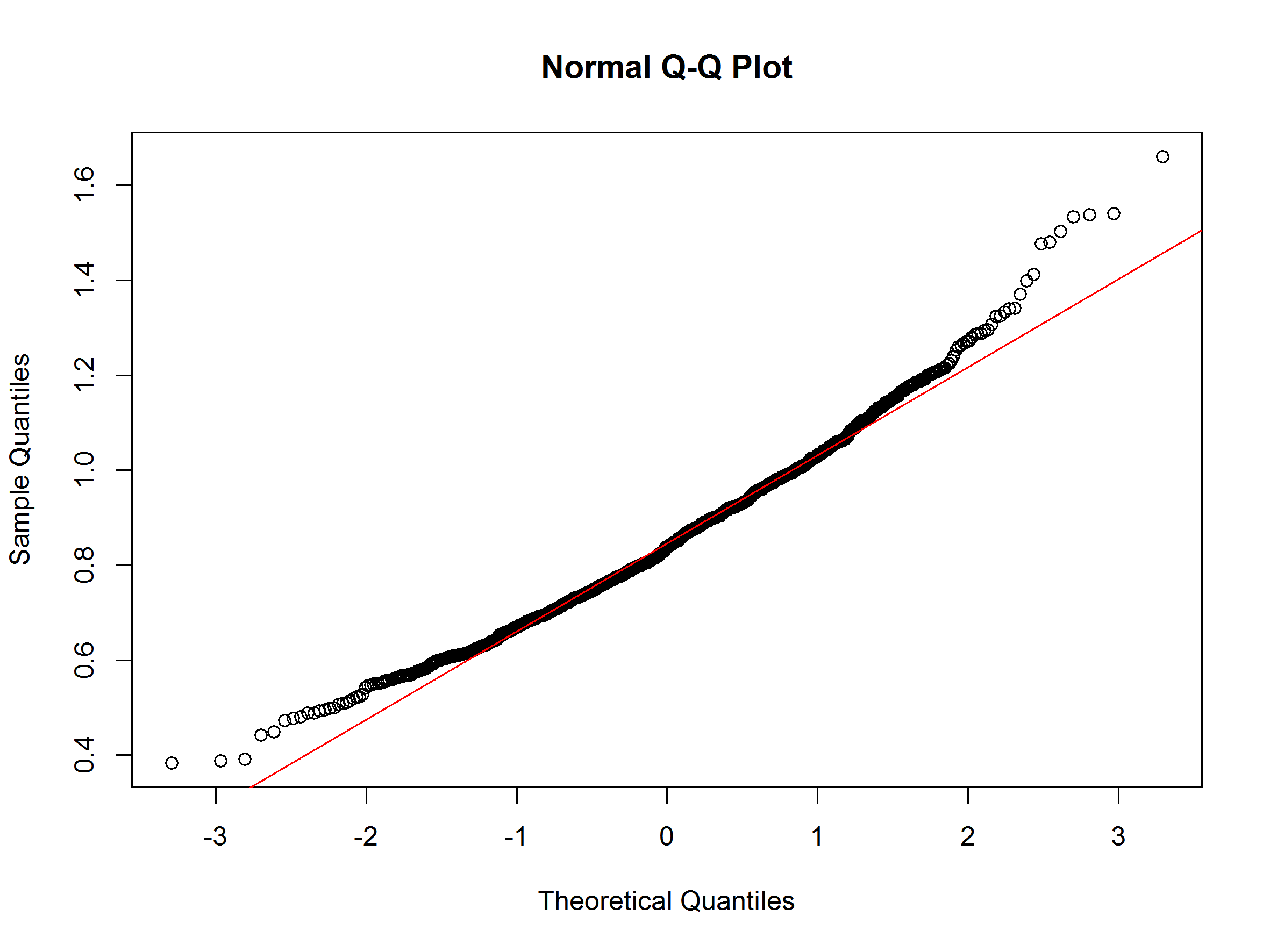}
	\hfill
	\includegraphics[width=0.475\textwidth, height=0.4\textheight]{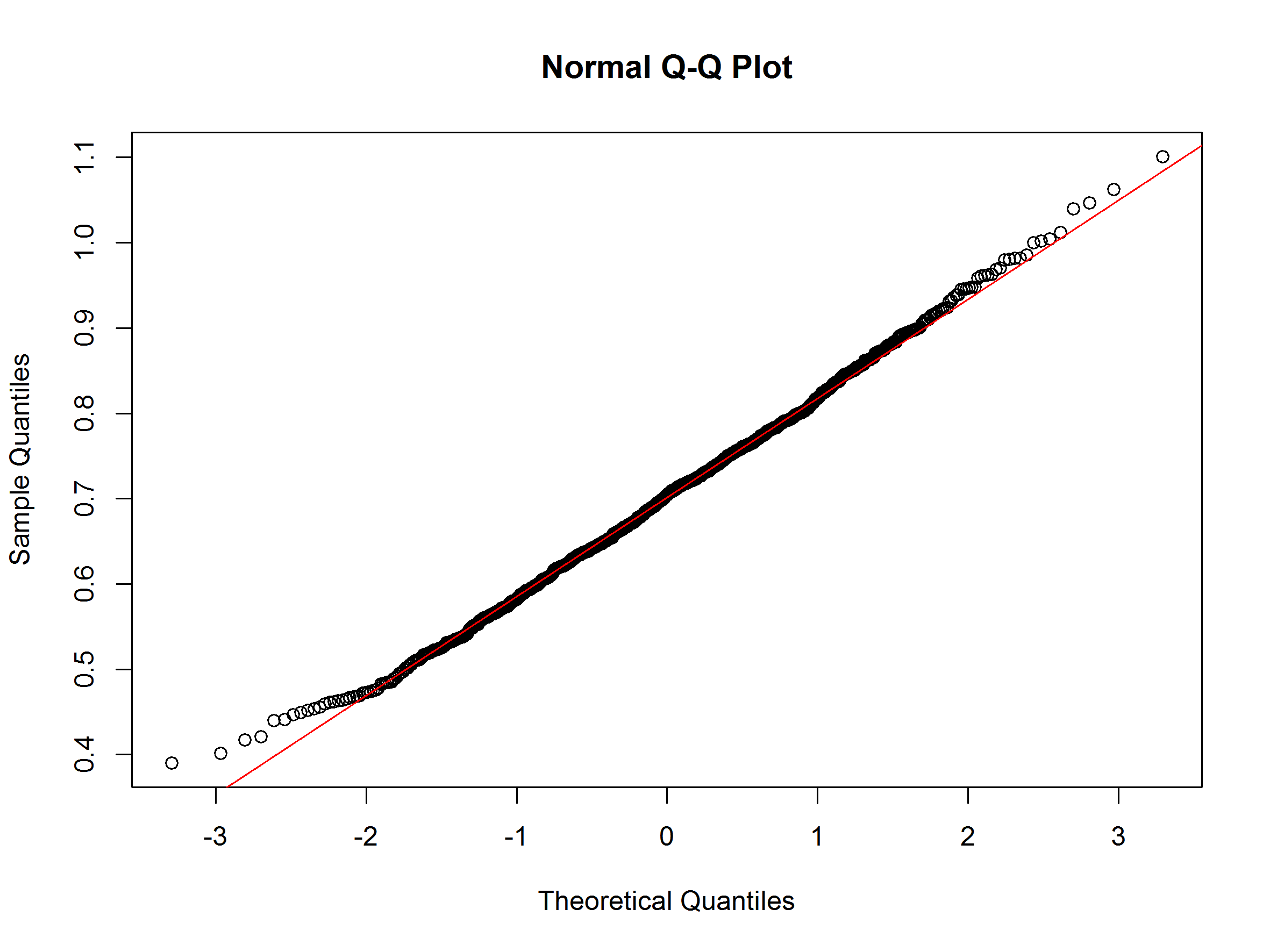}
	\caption{Normal QQ-Plot for Gini's Mean difference for $n=2000$ (left) and $n=5000$ (right) for strong dependence in the AR(1) process as well as in the EGARCH process. }
	\label{Gmd2}
\end{figure}

\begin{figure}
	\includegraphics[width=0.475\textwidth, height=0.4\textheight]{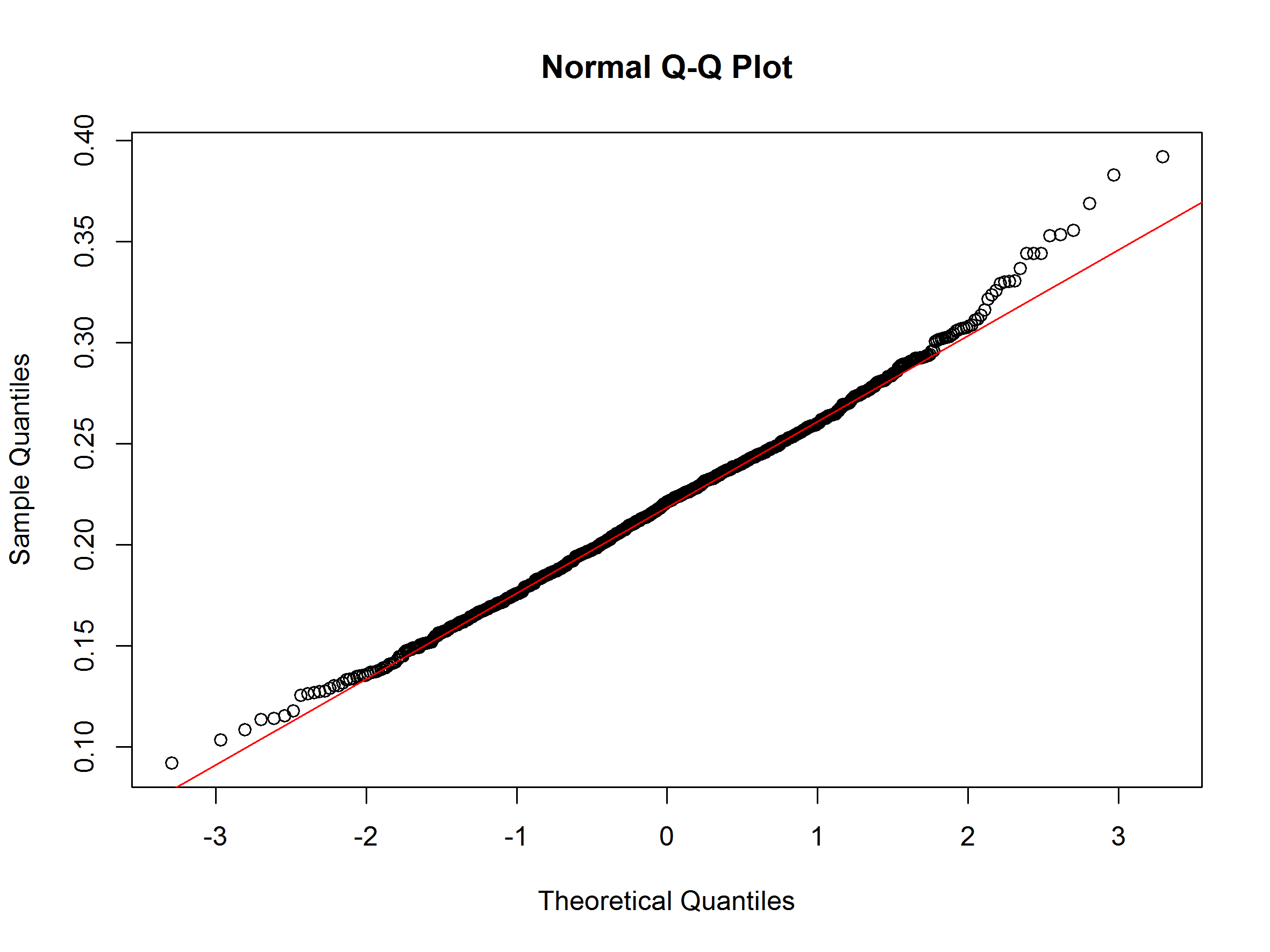}
	\hfill
	\includegraphics[width=0.475\textwidth, height=0.4\textheight]{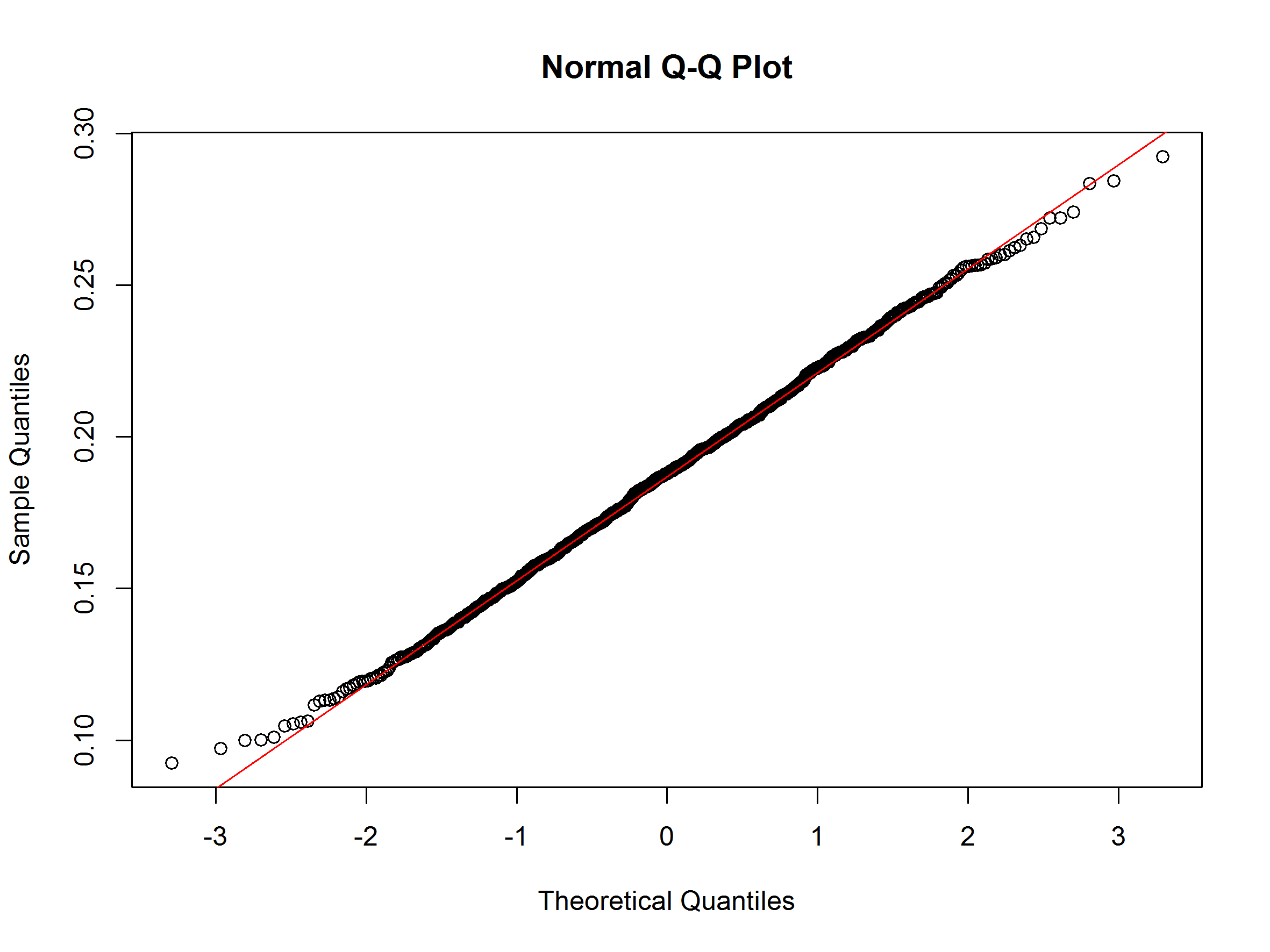}
	\caption{Normal QQ-Plot for the $LMS_n$-estimator for  $n=1000$ (bottom left) and $n=2000$ (bottom right) for strong dependence in the AR(1) process as well as in the EGARCH process. }
	\label{LMS2}
\end{figure}

\begin{figure}
	\includegraphics[width=0.475\textwidth, height=0.4\textheight]{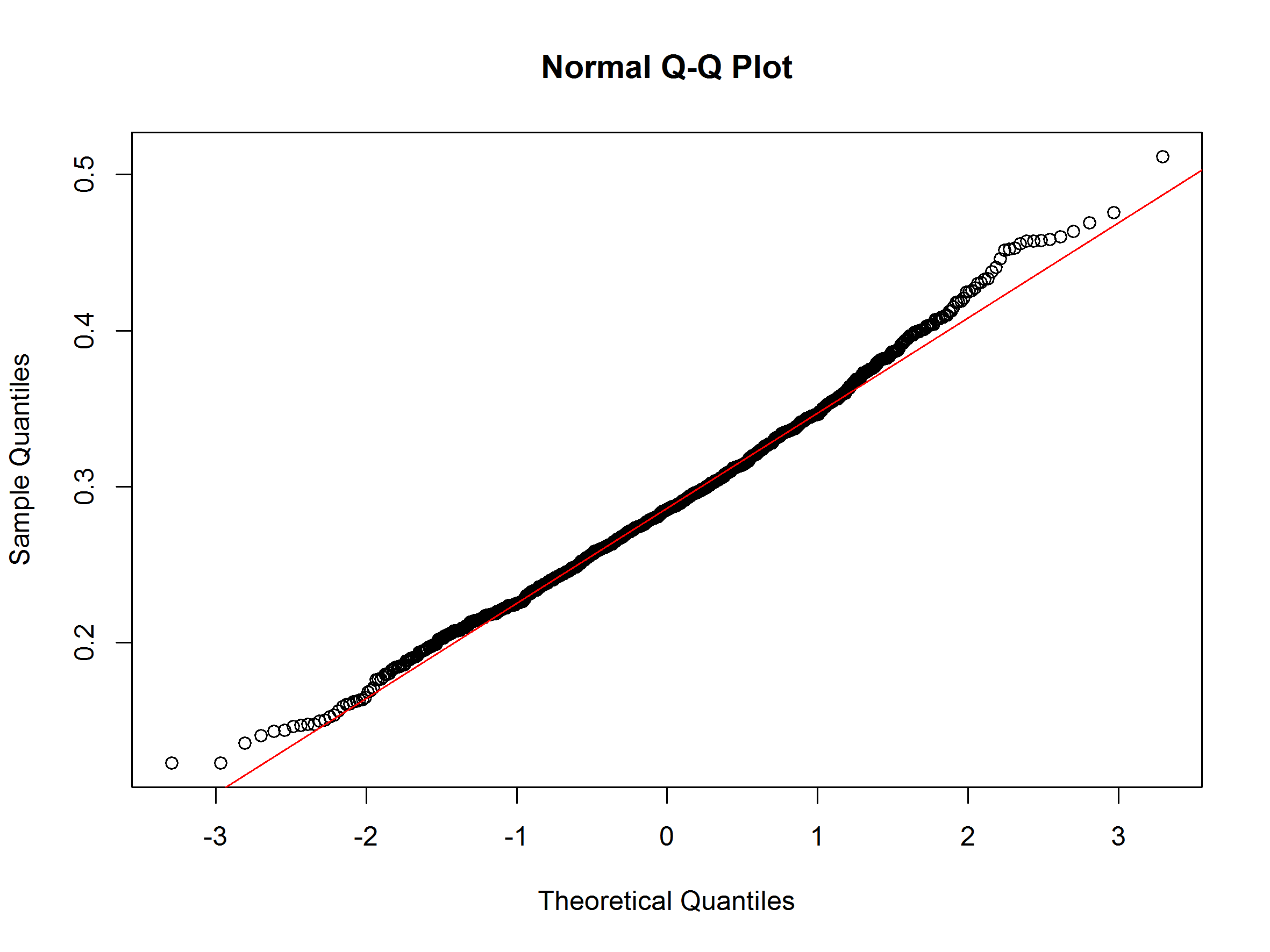}
	\hfill
	\includegraphics[width=0.475\textwidth, height=0.4\textheight]{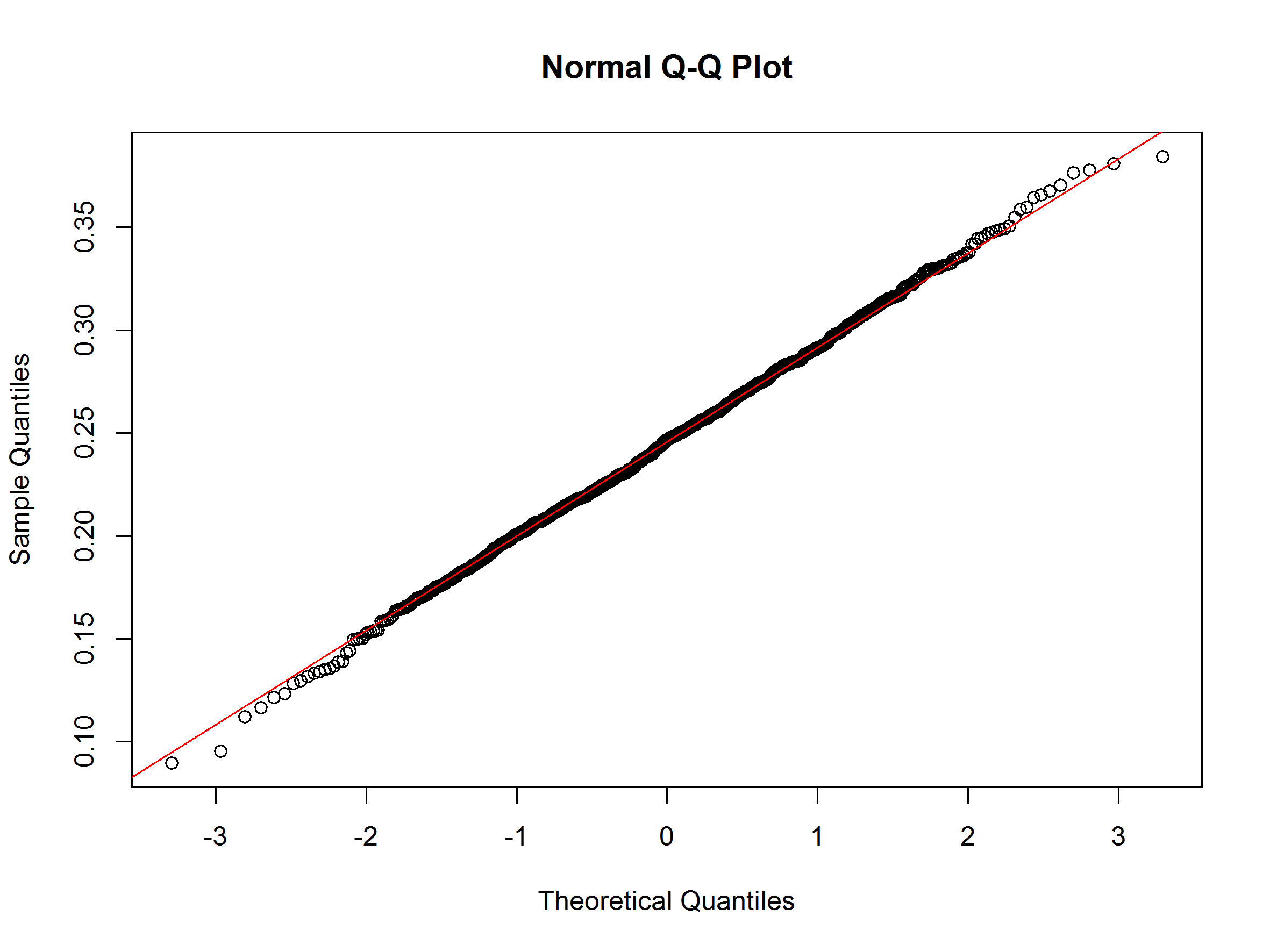}
	\caption{Normal QQ-Plot for the $Q$-estimator for  $n=1000$ (left) and $n=2000$ (right) for strong dependence in the AR(1) process as well as in the EGARCH process. }
	\label{Qest2}
\end{figure}

As expected the number of needed data in the sample to obtain a good approximation by the normal distribution increases with the dependence. Again, Gini's mean difference needs the largest sample ($n=5000$) for a good approximation. 


	
\section{Acknowledgements}
The financial support of the Deutsche Forschungsgemeinschaft (SFB 823, ”Statistical modelling
of non-linear dynamic processes”) is gratefully acknowledged. The author also would like to thank Marie Düker for her helpful suggestions concerning EGARCH-processes and Martin Wendler for many helpful discussions on this topic.

\nocite{*}
\newpage
\bibliography{LiteraturGL}

\appendix
	\section{Preliminary results}
	For the proofs of the main theorems some lemmata are needed. The following results are similar to the case of strong mixing (\cite{Fischer.2016}) but since the proofs need different arguments in some cases we state the proofs due to completeness.
	
	\vspace*{0.25cm}
	The first lemma is analogous to Lemma 4.2 in \cite{Fischer.2016} and an extension of Lemma 3.2.4 in \cite{Wen2011}.
	\begin{lemm}\label{covineq2}
		$\newline$
		Let $(X_n)_{n\in\mathbb{N}}$ be  NED with approximation constants $(a_l)_{l\in\mathbb{N}}$ on an absolutely regular process $(Z_n)_{n\in\mathbb{Z}}$ with mixing coefficients $(\beta(l))_{n\in\mathbb{N}}$. Moreover, let be $A_L=\sqrt{2\sum_{i=L}^\infty a_i}$ and let $h$ be bounded and satisfy the extendend variation condition. Then there exists for all $2\leq k \leq m $ a constant $C$, such that for
		
		$r=\max\left\{i_{(2)}-i_{(1)}, i_{(2k)}-i_{(2k-1)}\right\}$ with $i_{(1)}\leq\ldots\leq i_{(2k)}$ follows
		\begin{align*}
		\left\vert\mathbb{E}\left(g_k(X_{i_1},\ldots,X_{i_k})g_k(X_{i_{k+1}},\ldots,X_{i_{2k}})\right)\right\vert\leq C\left(\beta\left(\left[\frac{r}{3}\right]\right)+A_{\left[\frac{r}{3}\right]}\right). 
		\end{align*}
	\end{lemm}

	\begin{lemm}\label{finalcov}
		$\newline$
		Let the kernel $h$ be bounded and satisfy the extended variation condition.
		Moreover, let $(X_n)_{n\in\mathbb{N}}$ be NED with approximation constants $(a_l)_{l\in\mathbb{N}}$ on an absolutely regular process $(Z_n)_{n\in\mathbb{Z}}$ with mixing coefficients $(\beta(l))_{l\in\mathbb{N}}$ and $\sum_{l=0}^nl\left(\beta(l)+A_l\right)=O\left(n^{\gamma}\right)$ with $A_l=\sqrt{2\sum_{i=l}^\infty a_i}$ for a $\gamma>0$.
		Then for all $2\leq k\leq m$ and any constants $(c_{i,j})_{i,j\in \mathbb{N}}$
		\begin{align*}
		\sum_{i_1,\ldots,i_{2k}=1}^n\left|\mathbb{E}(g_k(X_{i_1},\ldots,X_{i_k})g_k(X_{i_{k+1}},\ldots,X_{i_{2k}})) c_{i_1,i_{k+1}}\right|= \max \limits_{i_1,i_{k+1}\in \lbrace 1,\ldots,n\rbrace}\vert c_{i_1,i_{k+1}}\vert O(n^{2k-2+\gamma}).
		\end{align*}
	\end{lemm}
	This lemma is analogous to Lemma 4.3 in \cite{Fischer.2016} and can be proved similar.

	\begin{proof}
		$\newline$
		Again set $\lbrace i_1,\ldots,i_{2k}\rbrace=\lbrace i_{(1)},\ldots,i_{(2k)}\rbrace$ with $i_{(1)}\leq \ldots \leq i_{(2k)}$. We can rewrite the above sum as
		\begin{align*}
		&\sum_{i_1,\ldots,i_{2k}=1}^n\left|\mathbb{E}(g_k(X_{i_1},\ldots,X_{i_k})g_k(X_{i_{k+1}},\ldots,X_{i_{2k}}))c_{i_1,i_{k+1}}\right|\\
		&\leq\max \limits_{i_1,i_{k+1}\in \lbrace 1,\ldots,n\rbrace}\vert c_{i_1,i_{k+1}}\vert\sum_{i_1,\ldots,i_{2k}=1}^n\left|\mathbb{E}(g_k(X_{i_1},\ldots,X_{i_k})g_k(X_{i_{k+1}},\ldots,X_{i_{2k}}))\right|\\
		&=\max \limits_{i_1,i_{k+1}\in \lbrace 1,\ldots,n\rbrace}\vert c_{i_1,i_{k+1}}\vert\sum_{l=0}^n\sum_{\stackrel{i_1,\ldots,i_{2k}=1}{\max\left\{ i_{(2)}-i_{(1)},i_{(2k)}-i_{(2k-1)}\right\}=l}}^n\left|\mathbb{E}(g_k(X_{i_1},\ldots,X_{i_k})g_k(X_{i_{k+1}},\ldots,X_{i_{2k}}))\right|\\
		&\leq \max \limits_{i_1,i_{k+1}\in \lbrace 1,\ldots,n\rbrace}\vert c_{i_1,i_{k+1}}\vert C \sum_{l=0}^n\sum_{\stackrel{i_1,\ldots,i_{2k},}{\max\lbrace i_{(2)}-i_{(1)},i_{(2k)}-i_{(2k-1)}\rbrace=l}} \left(\beta\left( \left[\frac{l}{3}\right]\right)+A_{\left[\frac{l}{3}\right]}\right),
		\end{align*}
		by application of Lemma \ref{covineq2}.
		
		Now we want to simplify the expression by using a calculation of the quantity of $(i_1,\ldots,i_{2k})$ where $\max\lbrace i_{(2)}-i_{(1)},i_{(2k)}-i_{(2k-1)}\rbrace=l$. Using combinatorical arguments we can see that there exist $(2k)!$ possibilities to gain the same sequence $i_{(1)},\ldots,i_{(2k)}$. We now fix $i_{(1)}$ and $i_{(2k)}$, having $n^2$ possibilities for this. Having in mind that $\max\lbrace i_{(2)}-i_{(1)},i_{(2k)}-i_{(2k-1)}\rbrace=l$ and suppose $i_{(2)}-i_{(1)}=\max\lbrace i_{(2)}-i_{(1)},i_{(2k)}-i_{(2k-1)}\rbrace=l$ then $i_{(2)}$ is automatically determined by the choice of $i_{(1)}$. Then, $i_{(2k-1)}$ can only take $l$ distinct values. Supposing $i_{(2k)}-i_{(2k-1)}=\max\lbrace i_{(2)}-i_{(1)},i_{(2k)}-i_{(2k-1)}\rbrace=l$ the same is valid. All remaining values of the $k$-tuple are arbitrary. Therefore, the quantity of the summands equals  $(2k)!\cdot n^2ln^{2k-4}=l\cdot(2k)!\cdot n^{2k-2}$ and
		
		\begin{align*}
		&\sum_{i_1,\ldots,i_{2k}=1}^n\left|\mathbb{E}(g_k(X_{i_1},\ldots,X_{i_k})g_k(X_{i_{k+1}},\ldots,X_{i_{2k}}))c_{i_1,i_{k+1}}\right|\\
		&\leq \max \limits_{i_1,i_{k+1}\in \lbrace 1,\ldots,n\rbrace}\vert c_{i_1,i_{k+1}}\vert C'n^{2k-2}\sum_{l=0}^nl\left(\beta\left( \left[\frac{l}{3}\right]\right)+A_{\left[\frac{l}{3}\right]}\right)=\max \limits_{i_1,i_{k+1}\in \lbrace 1,\ldots,n\rbrace}\vert c_{i_1,i_{k+1}}\vert O(n^{2k-2+\gamma}).
		\end{align*} 
		
	\end{proof}

	\begin{lemm}\label{remterm}
		$\newline$
		Let $h$ be a kernel satisfying the extended uniform variation condition, such that the $U$-distribution function $U$ is Lipschitz continuous. Moreover let 
		$(X_n)_{n\in\mathbb{N}}$ be NED with approximation constants $a_l=O\left(l^{-a}\right)$, where $a=\max\left\{\eta+3,12\right\}$, on an absolutely regular process $(Z_n)_{n\in\mathbb{Z}}$ with mixing coefficients $\beta(l)=O\left(l^{-\eta}\right)$ for a $\eta\geq 8$
		Then for all $2\leq k \leq m$ and $\gamma=\frac{\eta-3}{\eta+1}$ we have 
		\begin{align*}
		\sup\limits_{t\in\mathbb{R}}\big\vert\sum_{1\leq i_1,\ldots,i_k\leq n}g_k(X_{i_1},\ldots,X_{i_k},t)\big\vert=o(n^{k-\frac{1}{2}-\frac{\gamma}{8}})~\text{a.s.}.
		\end{align*}
	\end{lemm}

	\begin{coro}\label{glican}
		$\newline$
		Let $(X_n)_{n\in\mathbb{N}}$ be NED with approximation constants $(a_l)_{l\in\mathbb{N}}$ on an absolutely regular process $(Z_n)_{n\in\mathbb{Z}}$ with mixing coefficients $(\beta(l))_{l\in\mathbb{N}}$ with  $\sum_{l=1}^\infty l^2 \beta^{\frac{\gamma}{2+\gamma}}(l)< \infty$, $0<\gamma <1$. Additionally, let be  $\sum_{l=1}^\infty l^2 a_l^{\frac{\gamma}{2+2\gamma}} <\infty$ as well as $\sum_{l=1}^\infty l^2 \left(L\sqrt{2A_l}\right)^{\frac{\delta}{1+\delta}}<\infty$ for a $\delta>8$, where $L$ is the variation constant of the kernel $g(x_1,\ldots,x_m)=1_{\left[h(x_1,\ldots,x_m)\leq t\right]}$ and $A_L=\sqrt{2\sum_{i=l}^\infty a_i}$.
		Moreover let $h$ be a Lipschitz-continuous kernel with distribution function $H_F$ and related density $h_F<\infty$ and for all $2\leq k \leq m$ let $h_{F;X_2,\ldots,X_k}$ be bounded.
		Then
		\begin{align*}
		\sup\limits_{t\in\mathbb{R}}\left\vert\sqrt{n}\left(H_n(t)-H_F(t)\right)\right\vert=O_p(1).
		\end{align*}
	\end{coro}
	This corollary is a straightforward result from Theorem \ref{invpri} using Lemma \ref{finalcov}.
	The assumptions on the coefficients $a_l$ and $\beta_l$ of Lemma \ref{finalcov} are automatically fulfilled by the assumptions of Theorem \ref{invpri}.
	
	\begin{proof}
		$\newline$
		We again use the Hoeffding decomposition having
		\begin{align*}
		&\sup\limits_{t\in\mathbb{R}}\left\vert\sqrt{n}\left(H_n(t)-H_F(t)\right)\right\vert\\
		&=\sup\limits_{t\in\mathbb{R}}\left\vert\sqrt{n}\left(H_F(t)+\sum_{j=1}^{m}\binom{m}{j}\frac{1}{\binom{n}{j}}S_{jn,t}-H_F(t)\right)\right\vert\\
		&=\sup\limits_{t\in\mathbb{R}}\big\vert\frac{m}{\sqrt{n}}\sum_{i=1}^{n}g_1(X_i,t)+\sqrt{n}\frac{\binom{m}{2}}{\binom{n}{2}}\sum_{1\leq i<j\leq n}g_2(X_i,X_j,t)\\
		&+\ldots+\sqrt{n}\frac{1}{\binom{n}{m}}\sum_{1\leq i_1 <\ldots<i_m\leq n}g_m(X_{i_1},\ldots,X_{i_m},t)\big\vert\\
		&\leq \sup\limits_{t\in\mathbb{R}}\left\vert\frac{m}{\sqrt{n}}\sum_{i=1}^{n}g_1(X_i,t)\right\vert+\sup\limits_{t\in\mathbb{R}}\left\vert\sqrt{n}\frac{\binom{m}{2}}{\binom{n}{2}}\sum_{1\leq i<j\leq n}g_2(X_i,X_j,t)\right\vert\\
		&+\ldots+\sup\limits_{t\in\mathbb{R}}\left\vert\sqrt{n}\frac{1}{\binom{n}{m}}\sum_{1\leq i_1 <\ldots<i_m\leq n}g_m(X_{i_1},\ldots,X_{i_m},t)\right\vert.
		\end{align*}
		For the first term of the sum, using Theorem \ref{invpri} and the continuous mapping theorem, we get
		\begin{align*}
		\sup\limits_{t\in\mathbb{R}}\left\vert\frac{m}{\sqrt{n}}\sum_{i=1}^{n}g_1(X_i,t)\right\vert\rightarrow \lVert W \rVert_{\infty}.
		\end{align*}
		Since $W$ is a continuous Gaussian process we have $\lVert W \rVert_{\infty}=O_p(1)$.
		
		\vspace*{0.25cm}
		
		For the remaining results we want to apply Lemma \ref{remterm}. Therefore the kernel of the $U$-process $g(x_1,\ldots,x_m,t)=1_{[h(x_1,\ldots,x_m)\leq t]}$ has to satisfy the extended uniform variation condition. We use the Lipschitz continuity of $h$ for this. 
		\begin{align*}
		&\sup\limits_{\lVert(x_1,\ldots,x_m)-({X'}_1,\ldots,{X'}_m)\rVert\leq \epsilon }\left\vert 1_{[h({X'}_1,\ldots,{X'}_m)\leq t]}-1_{[h({x}_1,\ldots,{x}_m)\leq t]}\right\vert\\
		=&\begin{cases}
		1 &\text{, if } h({X'}_1,\ldots,{X'}_m) \in \left(t-L\epsilon,
		t+L\epsilon\right)\\
		0 &\text{, else}
		\end{cases}
		\end{align*}
		and so
		\begin{align*}
		&\mathbb{E}\left(\sup\limits_{\lVert(x_1,\ldots,x_m)-({X'}_1,\ldots,{X'}_m)\rVert\leq \epsilon }\left\vert1_{[h({X'}_1,\ldots,{X'}_m)\leq t]}-1_{[h({x}_1,\ldots,{x}_m)\leq t]}\right\vert\right)\\
		&\leq \sup\limits_{t \in \mathbb{R} }\left\vert\mathbb{E}\left(1_{[h({X'}_1,\ldots,{X'}_m)\in (t-L\epsilon,t+L\epsilon)]}\right)\right\vert\\
		&\leq \sup\limits_{t \in \mathbb{R} }\left\vert\int_{t-L\epsilon}^{t+L\epsilon}h_F(x)dx\right\vert\leq 2L\epsilon(\sup\limits_{x\in \mathbb{R}}h_F(x))\leq L'\epsilon,
		\end{align*}
		since $h_F$ is bounded.
		
		Using the arguments we can also show that $g$ satisfies the extended uniform variation condition.
		For arbitrary $2\leq k \leq m$ and $i_1<i_2<\ldots<i_m$
		\begin{align*}
		&\mathbb{E}\left(\sup\limits_{\lvert x_1-{Y}_{i_1}\rvert\leq \epsilon }\left\vert1_{[h({Y}_{i_1},X_{i_2},\ldots,X_{i_k},Y_{i_{k+1}},\ldots,{Y}_{i_m})\leq t]}-1_{[h({x}_1,X_{i_2},\ldots,X_{i_k},Y_{i_{k+1}},\ldots,{Y}_{i_m})\leq t]}\right\vert\right)\\
		&\leq \sup\limits_{t \in \mathbb{R} }\left\vert\int_{t-L\epsilon}^{t+L\epsilon}h_{F;X_{i_1},\ldots,X_{i_k}}(x)dx\right\vert\leq L\epsilon.
		\end{align*}

		Applying Lemma \ref{remterm} we get for $2\leq k \leq n$
		
		\begin{align*}
		\sup\limits_{t\in\mathbb{R}}\left\vert\sqrt{n}\frac{\binom{m}{k}}{\binom{n}{k}}\sum_{1\leq i_1,\ldots,i_k\leq n}g_k(X_{i_1},\ldots,X_{i_k},t)\right\vert\leq \sqrt{n} n^{-k} o_p(n^{k-\frac{1}{2}-\frac{\delta-2}{8\delta}})=o_p(n^{-\frac{\delta-2}{8\delta}}).
		\end{align*}

		\vspace*{0.1cm}
		With Slutsky's Theorem the proof is then completed.
	\end{proof}
	
	\vspace{0.3cm}
	
	We will now prove the convergence of the remaining term of the Bahadur representation (which was shown by \cite{Gho1971} under independence), where we need the following lemma.
	
	\begin{lemm}\label{varabsreg}
		$\newline$
		Let $(X_n)_{n\in \mathbb{N}}$ be NED with approximation constants $(a_l)_{l\in \mathbb{N}}$ on an absolutely regular process $(Z_n)_{n\in\mathbb{Z}}$ with mixing coefficients $(\beta(l))_{l\in\mathbb{N}}$. Moreover, let be $\sum_{l=1}^n a_l <\infty$ and $\sum_{l=1}^n \beta(l) <\infty$. If $X_1$ is bounded and $\mathbb{E}X_i=0$, then for a constant $C$
		\begin{align*}
		\mathbb{E}\left(\sum_{i=1}^n X_i\right)^2\leq C \cdot n.
		\end{align*}
	\end{lemm}
	
	\begin{proof}
		$\newline$
		Because of the stationarity of the $X_i$ we can write
		\begin{align}
		\mathbb{E}\left(\sum_{i=1}^n X_i\right)^2&\leq \sum_{\stackrel{1\leq i,k \leq n}{i+k\leq n}}\mathbb{E} \left\vert\left(X_iX_{i+k}\right)\right\vert \nonumber\\
		&\leq \sum_{i=1}^{n}\sum_{k=i+1}^{n-i} \left(4\left\Vert X_1 \right\Vert_\infty a_{\left[\frac{k}{3}\right]}+2\left\Vert X_1 \right\Vert^2_\infty \beta\left(\left[\frac{k}{3}\right]\right)\right) \label{vari}\\
		&n \sum_{k=1}^n \left(4\left\Vert X_1 \right\Vert_\infty a_{\left[\frac{k}{3}\right]}+2\left\Vert X_1 \right\Vert^2_\infty \beta\left(\left[\frac{k}{3}\right]\right)\right)\nonumber
		&\leq C\cdot n \nonumber,
		\end{align}
		where we used Lemma 2.18 of Borovkova et al. (2001) in line \ref{vari}.
	\end{proof}
	
	\begin{lemm}\label{bahad}
		$\newline$
		Let $(X_n)_{n\in\mathbb{N}}$ be NED with approximation constants $(a_l)_{l\in\mathbb{N}}$ on an absolutely regular process $(Z_n)_{n\in\mathbb{Z}}$ with mixing coefficients $(\beta(l))_{l\in\mathbb{N}}$ for which holds $\beta(l)=O\left(l^{-\delta}\right)$ and $a_l=O\left(l^{-\delta-2}\right)$ for a $\delta>1$.
		Moreover let $h(x_1,\ldots,x_m)$ be a Lipschitz-continuous kernel with distribution function $H_F$ and related density $0<h_F<\infty$ and for all $2\leq k \leq m$ let $h_{F;X_2,\ldots,X_k}$ be bounded. Then we have for the Bahadur representation with $\hat{\xi}_p=H_n^{-1}(p)$
		\begin{align*}
		\hat{\xi}_p=\xi_p+\frac{H_F(\xi_p)-H_n(\xi_p)}{h_F(\xi_p)}+o_p(\frac{1}{\sqrt{n}}).
		\end{align*}
	\end{lemm}
	
	\begin{proof}
		$\newline$
		Let us first define $t\in \mathbb{R}$, $\xi_{nt}=\xi_p+tn^{-\frac{1}{2}}, Z_n(t)=\sqrt{n}\frac{H_F(\xi_{nt})-H_n(\xi_{nt})}{h_F(\xi_p)}$ and $V_n(t)=\sqrt{n}\frac{H_F(\xi_{nt})-H_n(\hat{\xi}_{p})}{h_F(\xi_p)}$.

		We want to use that $\lvert p-H_n(\hat{\xi}_p)\rvert \leq \tfrac{1}{n}$ to obtain
		
		\begin{align*}
		V_n(t)&=\sqrt{n}\frac{H_F(\xi_{nt})-p+p-H_n(\hat{\xi}_{nt})}{h_F(\xi_p)}\\
		&=\underbrace{\sqrt{n}\frac{H_F(\xi_p+tn^{-\frac{1}{2}})-p}{h_F(\xi_p)}}_{=:V'_n(t)}+\underbrace{\sqrt{n}\frac{\overbrace{p-H_n(\hat{\xi}_p)}^{=O(n^{-1})}}{h_F(\xi_p)}}_{=O(n^{-\frac{1}{2}})}\longrightarrow t.
		\end{align*}
		
		\vspace*{0.3cm}
		
		The next step is to show that $Z_n(t)-Z_n(0)\stackrel{P}{\longrightarrow}0.$
		
		\vspace*{0.15cm}
		It is
		
		\begin{align*}
		&\var(Z_n(t)-Z_n(0))\\
		&= \frac{n}{h_F^2(\xi_p)}\var\left(\frac{1}{\binom{n}{m}}\sum_{1\leq i_1<\ldots<i_m\leq n}1_{\left[h(X_{i_1},\ldots,X_{i_m})\leq \xi_p+tn^{-\frac{1}{2}}\right]}-1_{\left[h(X_{i_1},\ldots,X_{i_m})\leq \xi_p\right]}\right).
		\end{align*}
		
		To find bounds for the right hand side, we define $U_n$ and $U_n'$ by 
		\begin{align*}
		U_n&=\frac{1}{\binom{n}{m}}\sum_{1\leq i_1<\ldots<i_m\leq n}1_{\left[h(X_{i_1},\ldots,X_{i_m})\leq \xi_p+tn^{-\frac{1}{2}}\right]}\\
		&=\theta+\sum_{j=1}^{m}\binom{m}{j}\frac{1}{\binom{n}{j}}\sum_{1\leq i_1<\ldots<i_j\leq n}g_k(X_{i_1},\ldots,X_{i_k})
		\end{align*}
		
		\begin{align*}
		U'_n&=\frac{1}{\binom{n}{m}}\sum_{1\leq i_1<\ldots<i_m\leq n}1_{\left[h(X_{i_1},\ldots,X_{i_m})\leq \xi_p\right]}\\
		&=\theta'+\sum_{j=1}^{m}\binom{m}{j}\frac{1}{\binom{n}{j}}\sum_{1\leq i_1<\ldots<i_j\leq n}g'_k(X_{i_1},\ldots,X_{i_k}).
		\end{align*}
		Therefore,
		\begin{align*}
		&\sqrt{\var\left(\frac{1}{\binom{n}{m}}\sum_{1\leq i_1<\ldots<i_m\leq n}1_{\left[\xi_p<h(X_{i_1},\ldots,X_{i_m})\leq \xi_p+tn^{-\frac{1}{2}}\right]}\right)}\\
		\leq& \sqrt{\underbrace{\var(\theta)}_{=0}}+\sqrt{\underbrace{\var(\theta')}_{=0}}+\sqrt{\var\left(\frac{m}{n}\sum_{i=1}^n(g_1(X_i)-g'_1(X_i))\right)}\\
		+&\sqrt{\var\left(\frac{\binom{m}{2}}{\binom{n}{2}}\sum_{1\leq i <j \leq n}g_2(X_i,X_j)\right)}+\sqrt{\var\left(\frac{\binom{m}{2}}{\binom{n}{2}}\sum_{1\leq i <j \leq n}g'_2(X_i,X_j)\right)}\\
		+&\ldots+\sqrt{\var\left(\frac{1}{\binom{n}{m}}\sum_{1\leq i_1<\ldots<i_m\leq n}g_m(X_{i_1},\ldots,X_{i_m})\right)}\\
		+&\sqrt{\var\left(\frac{1}{\binom{n}{m}}\sum_{1\leq i_1<\ldots<i_m\leq n}g'_m(X_{i_1},\ldots,X_{i_m})\right)}.\\
		\end{align*}
		We have already shown in Theorem \ref{asynom}, for all $2\leq k \leq m$,  that  $$\var\left(\frac{\binom{m}{k}}{\binom{n}{k}}\sum_{1\leq i_1<\ldots<i_k\leq n}g_k(X_{i_1},\ldots,X_{i_k})\right)=O(n^{-2+\gamma})$$ holds 
		for a $\gamma<1$, if the kernel is bounded and satisfies the extended variation condition. Analogous to the proof of Corollary \ref{glican} we can use that $g(x_1,\ldots,x_m)=1_{\left[h(X_{i_1},\ldots,X_{i_m})\leq \xi_p+tn^{-\frac{1}{2}}\right]}$ and $g'(x_1,\ldots,x_m)=1_{\left[h(X_{i_1},\ldots,X_{i_m})\leq \xi_p\right]}$  satisfy the extended variation condition. 
		
		Applying Lemma \ref{varabsreg} on $g_1(X_i)-g'_1(X_i)$ we have
		\begin{align*}
		\mathbb{E}\lvert\sum_{i=1}^n(g_1(X_i)-g'_1(X_i))\rvert^2\leq Cn.
		\end{align*}
		This is possible since $g_1, g'_1$ and $f(x,y)=x-y$ fulfil the variation condition and are bounded ($f$ when using bounded arguments). Also $\sum_{l=1}^\infty\beta(l)<\infty$ and $\sum_{l=1}^\infty a'_l<\infty$ because of the assumptions of Lemma \ref{bahad}.
		
		This helps us to gain
		\begin{align*}
		&\sqrt{\var\left(\frac{1}{\binom{n}{m}}\sum_{1\leq i_1<\ldots<i_m\leq n}1_{\left[\xi_p<h(X_{i_1},\ldots,X_{i_m})\leq \xi_p+tn^{-\frac{1}{2}}\right]}\right)}\\
		&\leq \sqrt{\frac{m^2}{n^2}Cn}+2(m-1)\sqrt{O(n^{-2+\gamma})}\leq \frac{Cm^2}{\sqrt{n}}+2(m-1)O(n^{-1+\gamma/2}),
		\end{align*}
		where the constant $C$ only depends on $\left\Vert g_1(X_i)- g'_1(X_i) \right\Vert_\infty$.
		
		Then
		
		\begin{align*}
		\var(Z_n(t)-Z_n(0))&\leq \frac{n}{h^2_F(\xi_p)}\left(\frac{Cm^2}{\sqrt{n}}+2(m-1)O(n^{-1+\gamma/2})\right)^2\\
		&\leq \frac{m^2}{h^2_F(\xi_p)} C^2 + \frac{4m^2(m-1)}{h^2_F(\xi_p)}C \sqrt{n}O(n^{-1+\gamma/2})+4(m-1)^2 O(n^{-2+\gamma})\\
		&\leq \frac{m^2}{h^2_F(\xi_p)} C^2+ \frac{4m^2(m-1)}{h^2_F(\xi_p)}C O(n^{-\frac{1}{2}+\gamma/2})+4(m-1)^2 O(n^{-2+\gamma}).
		\end{align*}
		
		Since $\lvert g_1(X_i)-g'_1(X_i)\rvert\leq 1$ for all $X_i$ and
		\begin{align*}
		&\lvert g_1(X_i)-g'_1(X_i)\rvert\stackrel{P}{\longrightarrow} 0
		\end{align*}
		the constant $C$ converges to zero in probability and therefore
		
		$$\var(Z_n(t)-Z_n(0))\stackrel{P}{\longrightarrow}0.$$
		
		\vspace*{0.1cm}
		
		Applying the Chebychev inequality we then have $Z_n(t)-Z_n(0)\stackrel{P}{\longrightarrow}0$ .
		
		Altogether we have for $t\in\mathbb{R}$ and every $\epsilon>0$
		\begin{align}
		\mathbb{P}(\sqrt{n}(\hat{\xi}_p-\xi_p)&\leq t, Z_n(0)\geq t+\epsilon)=\mathbb{P}(Z_n(t)\leq V_n(t),Z_n(0)\geq t+ \epsilon)\nonumber\\
		&\leq\mathbb{P}(\lvert Z_n(t)-Z_n(0)\rvert\geq\frac{\epsilon}{2})+\mathbb{P}(\lvert V_n(t)-t\rvert\geq \frac{\epsilon}{2})\longrightarrow 0 \nonumber
		\end{align}
		
		and analogously
		\begin{align*}
		\mathbb{P}(\sqrt{n}(\hat{\xi}_p-\xi_p)\geq t, Z_n(0)\leq t)\longrightarrow 0.
		\end{align*}
		
		Using Lemma 1 of \cite{Gho1971} the proof is completed. 
	\end{proof}

\end{document}